\documentclass[11pt,english]{article}

\usepackage[margin = 2.2cm]{geometry}

\usepackage{amsthm}
\usepackage{amsmath}
\usepackage{amssymb}
\usepackage{mathtools}
\usepackage{graphicx}
\usepackage[pdftex,colorlinks,backref=page,citecolor=blue,bookmarks=false]{hyperref}
\usepackage{color}
\usepackage{thm-restate}
\usepackage{enumitem}
\usepackage{subcaption}

\usepackage{caption}
\usepackage[square,sort,comma,numbers]{natbib} 
\usepackage{setspace}
\usepackage{cleveref}
\theoremstyle{plain}

\newtheorem*{theorem*}{Theorem}
\newtheorem{theorem}{Theorem}[section]
\crefname{theorem}{Theorem}{Theorems}
\Crefname{theorem}{Theorem}{Theorems}

\newtheorem*{lemma*}{Lemma}
\newtheorem{lemma}[theorem]{Lemma}
\crefname{lemma}{Lemma}{Lemmas}
\Crefname{lemma}{Lemma}{Lemmas}

\newtheorem*{claim*}{Claim}
\newtheorem{claim}[theorem]{Claim}
\crefname{claim}{Claim}{Claims}
\Crefname{claim}{Claim}{Claims}

\newtheorem*{innerclaim*}{Claim}

\crefname{claim}{Claim}{Claims}
\Crefname{claim}{Claim}{Claims}

\newtheorem{proposition}[theorem]{Proposition}
\crefname{proposition}{Proposition}{Propositions}
\Crefname{proposition}{Proposition}{Propositions}

\crefname{corollary}{Corollary}{Corollaries}
\Crefname{corollary}{Corollary}{Corollaries}

\crefname{conjecture}{Conjecture}{Conjectures}
\Crefname{conjecture}{Conjecture}{Conjectures}

\newtheorem{question}[theorem]{Question}
\crefname{question}{Question}{Questions}
\Crefname{question}{Question}{Questions}

\crefname{observation}{Observation}{Observations}
\Crefname{observation}{Observation}{Observations}

\crefname{example}{Example}{Examples}
\Crefname{example}{Example}{Examples}

\theoremstyle{definition}

\crefname{problem}{Problem}{Problems}
\Crefname{problem}{Problem}{Problems}

\crefname{definition}{Definition}{Definitions}
\Crefname{definition}{Definition}{Definitions}

\theoremstyle{remark}

\crefname{remark}{Remark}{Remarks}
\Crefname{remark}{Remark}{Remarks}

\usepackage{xpatch}
\xpatchcmd{\proof}{\itshape}{\normalfont\proofnamefont}{}{}
\newcommand{\proofnamefont}{}
\renewcommand{\proofnamefont}{\bfseries}

\usepackage{framed}

\newcommand{\remove}[1]{}

\newcommand{\ceil}[1]{
	\left\lceil #1 \right\rceil
}
\newcommand{\floor}[1]{
	\left\lfloor #1 \right\rfloor
}

\newcommand{\cP}{\mathcal{P}}

\newcommand{\cF}{\mathcal{F}}

\newcommand{\cC}{\mathcal{C}}

\newcommand{\eps}{\varepsilon}

\renewcommand{\setminus}{-}
\newcommand{\Ex}{\mathbb{E}}

\def \Xm {X^-}
\def \Zm {Z^-}

\def \Ym {Y^-}
\def \um {u^-}

\def \wm {w^-}
\def \wp {w^+}

\def \vm {v^-}

\def \xm {x^-}

\def \ym {y^-}

\def \zm {z^-}

\def \int {^\circ}

\def \mthree{^{-3}}

\def \nto {\leftarrow}

\begin{document}

\title{Pancyclicity of highly connected graphs}
\author{Shoham Letzter\thanks{
		Department of Mathematics, 
		University College London, 
		London WC1E~6BT, UK. 
		Email: \texttt{s.letzter}@\texttt{ucl.ac.uk}. 
		Research supported by the Royal Society.
	}
}

\date{}
\maketitle

\begin{abstract}

	\setlength{\parskip}{\medskipamount}
    \setlength{\parindent}{0pt}
    \noindent

	A well-known result due to Chvat\'al and Erd\H{o}s (1972) asserts that, if a graph $G$ satisfies $\kappa(G) \ge \alpha(G)$, where $\kappa(G)$ is the vertex-connectivity of $G$, then $G$ has a Hamilton cycle. We prove a similar result implying that a graph $G$ is \emph{pancyclic}, namely it contains cycles of all lengths between $3$ and $|G|$: if $|G|$ is large and $\kappa(G) > \alpha(G)$, then $G$ is pancyclic. This confirms a conjecture of Jackson and Ordaz (1990) for large graphs, and improves upon a very recent result of Dragani\'c, Munh\'a-Correia, and Sudakov.

\end{abstract}

\section{Introduction} \label{sec:intro}

	A graph $G$ is \emph{Hamiltonian} if it contains a Hamilton cycle, namely a cycle through all the vertices of $G$. The notion of Hamiltonicity is a key notion in combinatorics and computer science (see the surveys \cite{kuhn2014hamilton,gould2014recent}). The decision problem, of determining whether a given graph is Hamiltonian, is one of 21 problems that have been shown to be NP-complete in Karp's influential paper \cite{karp1972reducibility} from 1972. As such, it is interesting to find simple sufficient conditions for Hamiltonicity. Perhaps the earliest and best-known example is Dirac's theorem \cite{dirac1952some} from 1952, which asserts that every $n$-vertex graph with minimum degree at least $n/2$ (and $n \ge 3$) has a Hamilton cycle.

	We say that a graph $G$ is \emph{pancyclic} if it contains a cycle of length $\ell$, for every $\ell \in [3,|G|]$.
	In 1971 Bondy \cite{bondy1971pancyclic} showed that Dirac's theorem can be strengthened substantially, to show that every graph on $n$ vertices with minimum degree at least $n/2$ is either pancyclic, or a balanced complete bipartite graph. While interesting in its own right, Bondy believed that this example was a special case of a much more general phenomenon: in 1973 he \cite{bondy1973pancyclic} posed a famous `meta-conjecture', asserting that almost every non-trivial property implying Hamiltonicity also implies pancyclicity, up to a small number of exceptions.

	\def \ta {\tilde{\alpha}}
	Since then, the validity of this meta-conjecture has been proved in many cases. For example, Bauer and Schmeichel \cite{bauer1990hamiltonian} considered three properties involving degrees of vertices that were shown to imply Hamiltonicity by Bondy \cite{bondy1980longest}, Chvat\'al \cite{chvatal1972hamilton}, and Fan \cite{fan1984new}, and showed that they also imply pancyclicity, with the exception of some bipartite graphs. 
	Clark \cite{clark1981hamiltonian} (and later, independently, Shi \cite{shi1986connected} and Zhang \cite{zhang1989cycles}) showed that connected, locally connected claw-free graphs on at least three vertices are pancyclic, generalising a result of Oberly and Sumner \cite{oberly1979every}, who proved that such graphs are Hamiltonian. Another very recent illustration of the validity of Bondy's meta-conjecture is a result of Dragai\'c, Munh\'a-Correia, and Sudakov \cite{draganic2023generalization}. To state their result, let $\ta(G)$ be the \emph{bipartite independence number} of $G$, defined to be the maximum of $a+b$, over all $a,b$ such that there exist disjoint sets $A$ and $B$, of $a$ and $b$ vertices, such that $G$ has no $A$-$B$ edges. (In particular, $\ta(G) \ge \alpha(G)$ for every graph $G$.)
	They proved that if a graph $G$ satisfies $\delta(G) \ge \tilde{\alpha}(G)$, then $G$ is pancyclic, unless $G$ is a balanced complete bipartite graph, generalising a result of McDiarmid and Yolov \cite{mcdiarmid2017hamilton} who proved that this condition implies Hamiltonicity.

	The result of McDiarmid and Yolov is reminiscent of the following classic result of Chvat\'al and Erd\H{o}s \cite{chvatal1972note} from 1972: every graph $G$ satisfying $\kappa(G) \ge \alpha(G)$, where $\kappa(G)$ is the vertex-connectivity of $G$, is Hamiltonian. In light of the above discussion, it is natural to ask: is (almost) every graph $G$ satisfying $\kappa(G) \ge \alpha(G)$ pancyclic? It is easy to come up with a counterexample; indeed, any balanced complete bipartite graph is one. In fact, Bauer, van den Heuvel, and Schmeichel \cite{bauer1995toughness} found many examples of non-bipartite, triangle-free graphs $G$ satisfying $\kappa(G) = \alpha(G)$. Such graphs are counterexamples to the question, due to the triangle-freeness. Nevertheless, Lou \cite{lou1996chvatal} showed that every such graph $G$ contains a cycle of length $\ell$ for every $\ell \in [4, |G|]$, so the lack of triangles is the only barrier to pancyclicity for this class of examples.
	It is plausible that every graph $G$ which satisfies $\kappa(G) \ge \alpha(G)$ and contains a triangle is pancyclic. In this direction, in 1986 Jackson and Ordaz \cite{jackson1990chvatal} conjectured that every graph $G$ with $\kappa(G) > \alpha(G)$ is pancyclic. 

	A result of Erd\H{o}s \cite{erdos1974some} from 1974 implies that the stronger condition $\kappa(G) \ge 4(\alpha(G) + 1)^4$ guarantees pancyclicity. 
	Amar, Fournier, and Germa \cite{amar1991pancyclism} proved Jackson and Ordaz's conjecture when $\alpha(G) \in \{2,3\}$. Keevash and Sudakov \cite{keevash2010pancyclicity} proved the conjecture up to a constant factor: they showed that if $\kappa(G) \ge 600 \alpha(G)$ then $G$ is pancyclic. Very recently, Dragani\'c, Munh\'a-Correia, and Sudakov \cite{draganic2023chvatal} proved it asymptotically: they showed that for every $\eps > 0$, if $|G|$ is large enough and $\kappa(G) \ge (1 + \eps)\alpha(G)$, then $G$ is pancyclic.
	In this paper we prove Jackson and Ordaz's conjecture for sufficiently large graphs.

	\begin{theorem} \label{thm:main}
		There exists $n_0$ such that every graph $G$ on at least $n_0$ vertices, that satisfies $\kappa(G) > \alpha(G)$, is pancyclic.
	\end{theorem}
	
	We use many ingredients that are also used in \cite{draganic2023chvatal} in the proof of the asymptotic version of the conjecture, including lemmas from \cite{draganic2022pancyclicity,draganic2023chvatal} about extending and shortening paths by a small amount, the main result in \cite{draganic2022pancyclicity}, upper bounds on cycle-complete Ramsey numbers \cite{keevash2021cycle,erdos1978cycle}, and an upper bound on the Tur\'an numbers of even cycles \cite{bondy1974cycles}. Nevertheless, in each step of the proof, new ingredients are needed. Here we highlight one such ingredient, which is a so-called rotation-extension argument, and might have applications elsewhere.

	\paragraph{A new rotation argument.}
		A very useful tool in the study of cycles in graphs is the so-called \emph{rotation-extension} technique, introduced by P\'osa \cite{posa1976hamiltonian} in the study of Hamiltonicity of random graphs, and later and independently also used by Thomason \cite{thomason1978hamiltonian}.
		A standard rotation-extension argument proceeds as follows. Consider a cycle $C$, and for a vertex $u$ in $C$, denote by $\um$ the predecessor of $u$ in $C$ (according to an arbitrary orientation of $C$). If $u,v$ are two vertices in $C$, which can be joined by a path $P$ with non-empty interior that lies outside of $C$, and $\um \vm$ is an edge, then we can form a new cycle $(uC_{u \to \vm}\vm\um C_{\um \nto v}vxPyu)$ (see \Cref{fig:cycle1}) which extends $C$.
		Suppose now that our underlying graph has independence number $\alpha$.
		Then, using this argument, if $U$ is a set of at least $\alpha+1$ vertices in $C$, any two of whose vertices can be joined by a path as above, then we can find an appropriate extension of $C$ by the existence of an edge $\um \vm$.\footnote{This approach proves the Chvat\'al--Erd\H{o}s Hamiltonicity result mentioned above.}
		This approach fails, however, even when we can find such a $U$ of size only slightly below $\alpha$, as there is no way to guarantee the existence of an edge $\um \vm$. Our new idea here is to consider the set $\{\um, u\mthree : u \in U\}$, where $u\mthree = ((\um)^-)^-$. 
		If $|U| > \alpha/2$, and no two vertices of $U$ are too close on $C$, then this is a set of size larger than $\alpha$, which thus spans an edge. If that edge is of form $\um \vm$, $\um v\mthree$, or $u\mthree v\mthree$, with distinct $u,v \in U$, and $u,v$ can be joined by a path of length at least $6$ with interior outside of $C$, then we can extend $C$ similarly to the above (see \Cref{fig:cycle2,fig:cycle4,fig:cycle5}). The only other possibility is an edge of form $\um u\mthree$, which is a chord of length $2$. Under the right conditions, this either gives us the desired extension of $C$, or yields many non-intersection chords of length $2$, which can be utilised in other ways.

	\paragraph{Organisation of the paper.}
		In the next section, \Cref{sec:main-proof}, we state three lemmas, corresponding to three ranges of cycles lengths (upper, middle, and lower), and show how to prove our main result, \Cref{thm:main}, using these lemmas (and an additional result due to Dragani\'c, Munh\'a-Correia, and Sudakov \cite{draganic2022pancyclicity}).
		In \Cref{sec:prelims} we mention notation and preliminary results that will be used in the proofs of more than one of the three lemmas. We then prove the three lemmas in \Cref{sec:upper,sec:middle,sec:lower}, devoting one section to each lemma.
		We conclude the paper in \Cref{sec:conc} with some remarks regarding potential future research.

\section{Proof of the main theorem} \label{sec:main-proof}

	Our proof of \Cref{thm:main} splits into three lemmas, according to the length of the cycles they can guarantee. Here are these three lemmas.
	The notation $a \ll b$ means that $a$ is chosen to be sufficiently small with respect to $b$.

	\begin{restatable}[Upper range]{lemma}{lemUpperRange} \label{lem:upper-range}
		Let $0 < \delta \ll 1$ and let $n, \alpha, \ell \ge 1$ be integers, satisfying: $\alpha$ is sufficiently large, $\alpha \le n \le 4\alpha^2$, and $n / \delta \alpha \le \ell \le n$.
		Suppose that $G$ is a graph on $n$ vertices with $\alpha(G) = \alpha$ and $\kappa(G) > \alpha$.
		Then $G$ has a cycle of length $\ell$.
	\end{restatable}

	\begin{restatable}[Middle range]{lemma}{lemMiddleRange} \label{lem:middle-range}
		Let $0 < \delta \ll 1$ and let $n, \alpha, \ell \ge 1$ be integers, satisfying: $\alpha$ is sufficiently large, $\sqrt{n}/2 \le \alpha \le \delta n^{2/3}$, and $n/\alpha \le \ell \le \delta (n/\alpha)^2$. Suppose that $G$ is an $n$-vertex graph with $\alpha(G) = \alpha$ and $\kappa(G) \ge \alpha$. Then $G$ contains a cycle of length $\ell$.
	\end{restatable}

	\begin{restatable}[Lower range]{lemma}{lemLowerRange} \label{lem:lower-range}
		Let $0 < \delta \ll 1$ and let $n, \alpha, \ell \ge 1$ be integers, satisfying: $\alpha$ is sufficiently large, and $3 \le \ell \le \max\{n/\alpha, \delta \alpha\}$.
		Suppose that $G$ is an $n$-vertex graph with $\delta(G) > \alpha(G) = \alpha$.
		Then $G$ has a cycle of length $\ell$.
	\end{restatable}

	In addition, we shall need the following result, due to Dragani\'c, Munh\'a-Correia, and Sudakov \cite{draganic2022pancyclicity}, allowing us to assume that the independence number is relatively large in terms of the number of vertices.

	\begin{theorem}[Dragani\'c--Munha-Correia--Sudakov \cite{draganic2022pancyclicity}] \label{thm:draganic1}
		Let $\eps > 0$ and let $n$ be sufficiently large. Suppose that $G$ is a Hamiltonian graph on $n$ vertices, satisfying $n \ge (2 + \eps)\alpha(G)^2$. Then $G$ is pancyclic.
	\end{theorem}

	The proof of our main result, \Cref{thm:main}, follows directly from the last four results. 

	\begin{proof}[Proof of \Cref{thm:main}]
		Let $\delta$ be a constant satisfying $0 < \delta \ll 1$.
		Let $G$ be a graph on $n$ vertices, where $n$ is sufficiently large.
		Write $\alpha(G) = \alpha$, and let $\ell$ be an integer satisfying $3 \le \ell \le n$. We need to show that $G$ contains a cycle of length $\ell$.

		First note that $G$ is Hamiltonian, by the assumption $\kappa(G) > \alpha(G)$, and the well-known result of Chv\'atal and Erd\H{o}s \cite{chvatal1972note} mentioned in the introduction.
		By \Cref{thm:draganic1}, if $n \ge 4\alpha^2$ then $G$ is pancyclic, so in particular it contains a cycle of length $\ell$.

		So suppose that $n \le 4\alpha^2$, and notice that at least one of the following holds: $\ell \ge n/\delta \alpha$; $n/\alpha \le \ell \le \delta(n/\alpha)^2$ and $\sqrt{n}/2 \le \alpha \le \delta n^{2/3}$; or $3 \le \ell \le \max\{n/\alpha, \delta \alpha\}$. 
		Indeed, suppose that the first and third properties do not hold, namely $\max\{n/\alpha, \delta \alpha\} \le \ell \le n/\delta \alpha$. Then $\delta \alpha \le n/\delta \alpha$, implying that $\alpha \le \sqrt{n}/\delta \le \delta n^{2/3}$. Thus also $n/\alpha \le \ell \le n/\delta \alpha \le \delta(n/\alpha)^2$ (for the last inequality, notice that $n$ is large enough and $n \le 4\alpha^2$, implying that $n/\alpha$ is large enough). This shows that the second property holds. Hence, we can apply one of \Cref{lem:upper-range,lem:middle-range,lem:lower-range} to show that $G$ contains a cycle of length $\ell$.
	\end{proof}

\section{Preliminaries} \label{sec:prelims}

	We will prove each of \Cref{lem:upper-range,lem:middle-range,lem:lower-range} in a separate section, and mention the relevant preliminaries in these sections. Here we first define some notation, and then mention a few preliminaries that will be used in more than one of these sections.

\paragraph{Chords.}
	A \emph{chord} in a path $P$ is an edge joining two non-consecutive vertices in $P$. The \emph{length} of a chord is the length of the subpath of $P$ between the two vertices of the chord. We call two chords \emph{non-intersecting} if the interiors of the subpaths of $P$ between the two vertices of each chord are disjoint. We will also use these notions with respect to a cycle $C$ (when talking about the subpath of $C$ between the vertices of the chord, we mean the shorter of these two subpaths, and in fact we will only consider chords of length $2$ or $3$). See \Cref{fig:chords} for an illustration of this notion.
	\begin{figure}[h]
		\centering
		\includegraphics[scale = .7]{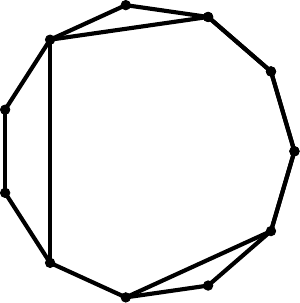}
		\caption{A cycle with three non-intersecting chords: two of length $2$ and one of length $3$}
		\label{fig:chords}
	\end{figure}
	
\paragraph{Predecessors in cycles.}
	Some of our proofs use arguments involving rotations, and as such the following notation will be useful.
	Given a cycle $C$ with an arbitrary direction, and a vertex $u$ in $C$, we denote by $\um$ the predecessor of $u$ in $C$ (according to the arbitrary direction of $C$). More generally, we will write $u^{-i}$ for the vertex in $C$ obtained by taking $i$ steps backwards from $u$.
	Additionally, given two vertices $u,v$ in $C$, we write $C_{u \to v}$ for the subpath of $C$ with ends $u,v$ which contains the successor of $u$, and write $C_{u \nto v}$ for the other subpath of $C$ with ends $u,v$. Similarly, for a path $P$ and vertices $u,v$ in $P$, we write $P_{u \to v}$ for the subpath of $P$ with ends $u,v$.

\paragraph{More path and cycle notation.}
	For a path $P$ we write $|P|$ for its \emph{order}, namely its number of vertices, and $\ell(P)$ for its \emph{length}, namely its number of edges.
	As the order and length of a cycle are the same, we normally stick to the more standard notation $|C|$ for its length.
	For distinct vertices $v_1, \ldots, v_{\ell}$, denote by $v_1\ldots v_{\ell}$ the path with edges $v_1v_2, \ldots, v_{\ell-1}v_{\ell}$, and by $(v_1 \ldots v_{\ell})$ the cycle with edges $v_1v_2, \ldots, v_{\ell-1}v_{\ell}, v_{\ell}v_1$.
	For vertices $x,y$, an \emph{$xy$-path} is a path with ends $x$ and $y$.
	If $P$ is an $xy$-path and $Q$ is a $yz$-path, and $P$ and $Q$ share only the vertex $y$, then $xPyQz$ denotes the path which is the concatenation of $P$ and $Q$.
	Similarly, if $P$ and $Q$ are $xy$-paths that do not share vertices other than $x$ and $y$, then $(xPyQx)$ is the cycle which is the union of $P$ and $Q$.
	We sometimes use $PQ$ and $(PQ)$ for the same things if the ends are not explicitly known.
	We denote by $P_k$ the path on $k$ vertices.

\subsection{Paths with many short chords}
	The next lemma finds a path with many short, non-intersecting chords.

	\begin{lemma} \label{lem:path-chords}
		Suppose that $G$ is a graph satisfying $\delta(G) > \alpha(G)$.
		Then for every integer $k \le \alpha(G)/6$, there is a path of length at most $3k$ which has $k$ pairwise non-intersecting chords of length $2$ or $3$, at least one of which has length $2$.
	\end{lemma}

	\begin{proof}
		Write $\alpha = \alpha(G)$.
		We define paths $P_1, \ldots, P_k$ such that the following holds for $i \in [k]$: $|P_i| \le 3i$; and $P_i$ has $i$ pairwise non-intersecting chords of length $2$ or $3$, at least one of which with length $2$.
		To start, notice that $G$ contains a triangle (as the neighbourhood of any vertex has size at least $\alpha + 1$ and is thus not independent). Let $x,y,z$ be the vertices of a triangle, define $P_1 = xyz$, and notice that $xz$ is a chord of length $2$ in $P_1$. 
		Now suppose that $P_1, \ldots, P_i$ are defined and $i \in [k-1]$, and let $x$ be one of $P_i$'s ends. Consider the graph $G' = G - (V(P_i) - \{x\})$. If $x$ is in a triangle in $G'$, let $y, z$ be the other two vertices in such a triangle, define $P_{i+1} = P_i xyz$, and observe that $xz$ is a chord of length $2$ that does not intersect the chords of $P_i$.
		Similarly, if $x$ has a neighbour $y$ which is in a triangle $yzw$ in $G'$ (with $z,w \neq x$), define $P_{i+1} = P_i xyzw$, and notice that $yw$ is a chord of length $2$.
		It remains to consider the case where $x$ and one of its neighbours $y$ are not in a triangle in $G$. Then the sets $N_x := N_{G'}(x) - \{y\}$ and $N_y := N_{G'}(y) - \{x\}$ are independent and disjoint. Notice that $|N_x|, |N_y| > \alpha - 3k \ge \alpha/2$. Thus $N_x \cup N_y$ is a set of size larger than $\alpha$, showing that it contains an edge $zw$, which must have one end in $N_x$ and the other in $N_y$; say $z \in N_x$ and $w \in N_y$. Define $P_{i+1} = P_i xzwy$, and notice that $xy$ is a chord of length $3$ which does not intersect any chords in $P_i$.
		This completes the proof that a sequence $P_1, \ldots, P_k$ with the required properties exists.
		The path $P_k$ satisfies the requirements of the lemma.
	\end{proof}

	A path $P$ as in \Cref{lem:path-chords} has the useful property that it can be shortened by any amount up to $k$. This is formalised in the next proposition.

	\begin{proposition} \label{prop:chords}
		Let $P$ be a path with ends $x,y$ and $k$ non-intersecting chords of length $2$ or $3$, with at least one of length $2$.
		Then, for every $k' \le k$, there is an $xy$-path $P'$ with $V(P') \subseteq V(P)$ and $|P'| = |P| - k'$.
	\end{proposition}

	\begin{proof}
		Consider a collection of $k$ non-intersecting chords of length $2$ or $3$ in $P$, at least one of which having length $2$, and denote the number of chords of length $2$ in this collection by $a$ and the number of chords of length $3$ by $b$.
		We claim that there exist $a',b'$ such that $0 \le a' \le a$, $0 \le b' \le b$, and $a' + 2b' = k'$.
		Indeed, taking $b' = \min\{\floor{k'/2}, b\}$ and $a' = k' - 2b'$ works. Let $e_1, \ldots, e_{a'+b'}$ be non-intersecting chords, $a'$ of which of length $2$ and the rest of length $3$, and consider the path $P'$ obtained by replacing the subpath of $P$ between the ends of $e_i$ by $e_i$ itself, for $i \in [a'+b']$. Then $|P'| = |P| - a' - 2b' = |P| - k'$.
	\end{proof}

\subsection{Cycle-complete Ramsey numbers}
	We now mention two results regarding cycle-complete Ramsey numbers. These refer to the Ramsey number $r(C_{\ell}, K_s)$, namely the minimum $n$ such that every $n$-vertex graph contains either a cycle of length $\ell$ or an independent set of size $s$.
	The first is an early result about these numbers, which tends to be useful when $\ell$ is small.

	\begin{theorem}[Erd\H{o}s--Faudree--Rousseau--Schelp \cite{erdos1978cycle}] \label{thm:ramsey-erdos}
		Let $\ell \ge 3$ and $s \ge 2$, and write $x = \floor{\frac{\ell-1}{2}}$. Then 
		\begin{equation*}
			r(C_{\ell}, K_s) \le \left( (\ell-2)(s^{1/x} + 2) + 1 )\right)(s-1).
		\end{equation*}
	\end{theorem}

	The next result determines $r(C_{\ell}, K_s)$ precisely when $\ell$ is large with respect to $s$, and resolves a conjecture from \cite{erdos1978cycle}. 

	\begin{theorem}[Keevash--Long--Skokan \cite{keevash2021cycle}] \label{thm:ramsey-keevash}
		There is a constant $c \ge 1$ such that the following holds for $s \ge 3$ and $\ell \ge \frac{c \log s}{\log \log s}$.
		\begin{equation*}
			r(C_{\ell}, K_s) = (\ell-1)(s-1) + 1.
		\end{equation*}
	\end{theorem}

\section{Upper range} \label{sec:upper}

	Our aim in this section is to prove \Cref{lem:upper-range}, restated here, which guarantees the existence of long cycles.
	\lemUpperRange*

	This section is the one where our proof differs the most from \cite{draganic2023chvatal}. Key new components here are a new rotation-extension technique, and the use of rotation-extension arguments given a cycle $C$ for which there are no long paths outside of $C$.

	The proof proceeds as follows.
	First, we apply \Cref{lem:path-chords} to find a path $P_0$ with many non-intersecting chords of length $2$ or $3$, which can be extended to a cycle of length at most $\ell$. Using \Cref{lem:long} below, we show that there is a cycle $C$ that either extends $P_0$ or contains many non-intersecting chords of length $2$, which is not much longer than $\ell$, such that either $G - V(C)$ is $P_5$-free, or $|C| \ge \ell$.
	If $|C| < \ell$, then $G - V(C)$ is $P_5$-free, and we use \Cref{lem:length-3-remainder} to find a cycle of length at most $\ell$. Otherwise, we use two lemmas from \cite{draganic2022pancyclicity,draganic2023chvatal} (\Cref{lem:shortening-min-deg,cor:shortening-indep} below) to shorten $C$ gradually to have length slightly more than $\ell$, so that it can further be shortened to have length exactly $\ell$ using the short chords.

	\subsection{Lemmas}
		In this subsection we state the lemmas we will use in the proof of \Cref{lem:upper-range}.
		Most of the work towards proving \Cref{lem:upper-range} will go into the proof of the following lemma. 

		\begin{lemma} \label{lem:long}
			Let $0 < \delta \ll 1$, and let $n, \alpha, \ell \ge 1$ be integers such that $n$ is large.

			Let $G$ be a graph on $n$ vertices with $\alpha(G) = \alpha$ and $\kappa(G) > \alpha$.
			Let $P_0$ be a path of length at most $\delta \alpha$, and suppose that there is a cycle of length at most $\ell$ that contains $P_0$.
			Then there is a cycle $C$ satisfying the following three properties.
			\begin{enumerate}[label = \rm(\arabic*)]
				\item \label{itm:long-jump}
					$|C| \le \ell + \frac{n}{\delta \alpha}$.
				\item \label{itm:long-length-4}
					Either $G - V(C)$ has no paths of length $4$, or $|C| \ge \ell$. 
				\item \label{itm:long-chords}
					Either $P_0 \subseteq C$, or $C$ has at least $\delta \alpha$ pairwise non-intersecting chords of length $2$.
			\end{enumerate}
		\end{lemma}
		 
		Given a cycle as guaranteed by the last lemma, if its length is less than $\ell$, then we can apply the following lemma to find a cycle of length exactly $\ell$.

		\begin{lemma} \label{lem:length-3-remainder}
			Let $\ell, n \ge 1$ be integers satisfying $3 \le \ell \le n$. 
			Let $G$ be a graph on $n$ vertices, satisfying $\kappa(G) > \alpha(G)$. Suppose that $C_0$ is a cycle in $G$ satisfying: $|C_0| < \ell$; there are at least 18 pairwise non-intersecting chords of length $2$ or $3$ in $C_0$; and $G - V(C_0)$ is $P_5$-free.
			Then $G$ has a cycle of length exactly $\ell$.
		\end{lemma}

		Otherwise, we apply one of the following results, due to Draganic--Munh\'a-Correia--Sudakov \cite{draganic2022pancyclicity,draganic2023chvatal}, allowing us to shorten a given path by a relatively small amount. 

		\begin{lemma}[Lemma 2.9 in \cite{draganic2023chvatal} (a Consequence of Proposition 2.9 in \cite{draganic2022pancyclicity})] \label{cor:shortening-indep}
			Let $G$ be an $n$-vertex graph with independence number at most $\alpha$, and let $P$ be a path in $G$ with ends $x,y$, satisfying $|P| > 4\alpha$.
			Then there is an $xy$-path $P'$ such that $|P| - \ceil{\frac{20\alpha^2}{|P|}} \le |P'| < |P|$.
		\end{lemma}

		\begin{lemma}[Lemma 2.8 in \cite{draganic2023chvatal}] \label{lem:shortening-min-deg}
			Let $G$ be an $n$-vertex graph with minimum degree $\delta$, and let $P$ be a path with ends $x,y$ and length at least $20n/\delta$. Then there is an $xy$-path $P'$ satisfying $|P| - 20n/\delta \le |P'| < |P|$.
		\end{lemma}

	\subsection{Proof of upper range lemma}

		We now prove the upper range lemma, using the results mentioned in the previous subsection.

		\begin{proof}[Proof of \Cref{lem:upper-range}]
			Let $\eta$ be a constant satisfying $\delta \ll \eta \ll 1$.
			Suppose that $\ell$ satisfies $\frac{n}{\delta \alpha} \le \ell \le n$ and write $k = \min\{\eta \alpha, \ell/6\}$.
			By \Cref{lem:path-chords}, there is a path $P_0$ of length at most $3k$ which has $k$ non-intersecting chords of length $2$ or $3$, with at least one of length $2$. Let $C_0$ be a shortest cycle that extends $P_0$. 

			We claim that $|C_0| \le \ell$.
			Indeed, denote by $x,y$ the ends of $P_0$, and consider the graph $G_0$, obtained from $G$ by removing the interior vertices of $P_0$. Then $\kappa(G_0) \ge \kappa(G) - 3k \ge \alpha/2$. Thus, there are $\alpha/2$ paths from $x$ to $y$ in $G_0$ with non-intersecting interiors. Let $Q_0$ be a shortest $xy$-path; then $|Q_0| \le n/(\alpha/2) = 2n/\alpha \le \ell/2$, implying that $|C_0| \le \ell$, as claimed.

			Let $C_1$ be a cycle as guaranteed by \Cref{lem:long}, applied with $P_0$ and $C_0$ and $\delta_{\ref{lem:long}} = \eta$. If $|C_1| < \ell$, then by \ref{itm:long-length-4}, the graph $G - V(C_1)$ is $P_5$-free. Additionally, by \ref{itm:long-chords}, by choice of $P_0$, and by $\alpha$ being large, $C_1$ has $18$ non-intersecting chords of length $2$ or $3$. \Cref{lem:length-3-remainder} thus yields a cycle of length exactly $\ell$, as required. So, we may assume that $|C_1| \ge \ell$.

			Suppose now that $P_0 \subseteq C_1$, let $P_1$ be the subpath of $C_1$ obtained by removing the edges of $P_0$, denote its ends by $x,y$, and let $G_1$ be the graph obtained from $G$ by removing the interior vertices of $P_0$. Then $\delta(G_1) \ge \alpha/2$ and $\alpha(G_1) \le \alpha$. Write $\ell' = \ell - \ell(P_0)$. Let $P_2$ be the shortest $xy$-path of length at least $\ell'$ in $G_1$.
			\begin{claim}
				$\ell(P_2) \le \ell' + k$.
			\end{claim}

			\begin{proof}
				Suppose that $\ell(P_2) > \ell' + k$. 
				By \Cref{lem:shortening-min-deg}, there is an $xy$-path $P_2'$ in $G_1$ that satisfies $\ell(P_2) - r \le \ell(P_2') < \ell(P_2)$, where $r = \frac{20|G_1|}{\delta(G_1)} \le \frac{40n}{\alpha}$. If $k \ge \frac{40n}{\alpha}$, then $\ell(P_2') \ge \ell' + k - r \ge \ell'$, contradicting the minimality of $P_2$. 

				We thus have $k \le \frac{40n}{\alpha}$. In particular, $k = \eta \alpha$ (because $\ell/6 \ge \frac{40n}{\alpha}$), $\alpha^2 \le \frac{40n}{\eta}$, and $\frac{\ell}{\alpha} \ge \frac{n}{\delta \alpha^2} \ge \frac{\eta}{40\delta} \ge 10$ (by $\delta \ll \eta$).
				Note that $\ell(P_2) \ge \ell - \ell(P_0) \ge \ell/2 \ge 5\alpha$ (using the last inequality). Hence, by \Cref{cor:shortening-indep}, there exists an $xy$-path $P_2'$ in $G_1$ satisfying $\ell(P_2) - r \le \ell(P_2') < \ell(P_2)$, where $r = \ceil{\frac{20\alpha^2}{|P_2|}} \le \frac{20\alpha^2}{\ell/2} + 1 = \frac{40\alpha^2}{\ell} + 1 \le \frac{1600\delta \alpha}{\eta} + 1 \le \eta \alpha = k$ (using $\frac{\ell}{\alpha} \ge \frac{\eta}{40\delta}$, $\delta \ll \eta$, and $\alpha$ being large). Thus $\ell(P_2') \ge \ell(P_2) - r \ge \ell'$, again contradicting the minimality of $P_2$.
			\end{proof}

			Recall that $P_0$ has $k$ non-intersecting chords of length $2$ or $3$, with at least one of length $2$. Then, by \Cref{prop:chords}, there is an $xy$-path $P_0'$, with $V(P_0') \subseteq V(P_0)$ and $\ell(P_0') = \ell - \ell(P_2)$ (indeed, we have $\ell(P_0) - k \le \ell - \ell(P_2) \le \ell(P_0)$ by the above claim and choice of $P_2$). The concatenation of $P_0'$ and $P_2$ is a cycle of length $\ell$, as required.

			It remains to consider the case $P_0 \not\subseteq C_1$. Then $C_1$ has at least $\eta\alpha$ non-intersecting chords of length $2$, and $|C_1| \le \ell + \frac{n}{\eta \alpha}$.

			Suppose first that $n \le \eta^2\alpha^2$. Write $t = |C_1| - \ell$; then $t \le \eta\alpha$ using the last condition above. Pick $t$ non-intersecting chords of length $2$ in $C_1$, and, for each of these chords $e$, replace the subpath of length $2$ of $C_1$ between the vertices of $e$ by $e$ itself. This yields a cycle of length exactly $\ell$.

			Now consider the case $n \ge \eta^2 \alpha^2$. Let $P_1$ be a subpath of $C_1$ of length $\ell/2$ with a maximum number of non-intersecting chords of length $2$, and let $P_2$ be the subpath of $C$ obtained by removing the edges of $P_1$. 
			Noting that $|C| \le \ell + \frac{n}{\eta \alpha} \le 2\ell$, we conclude that $P_1$ has at least $\eta\alpha/4$ non-intersecting chords of length $2$.
			Denote the ends of $P_2$ by $x,y$, and let $P_3$ be a shortest $xy$-path in $G[V(P_2)]$ of length at least $\ell/2$. We claim that $\ell(P_3) \le \ell/2 + \eta \alpha/2$.
			If not, then by \Cref{cor:shortening-indep} there is an $xy$-path $P_3'$ in $G[V(P_3)]$ with $\ell(P_3) - r \le \ell(P_3') < \ell(P_3)$, where $r = \ceil{\frac{20\alpha^2}{\ell}} \le \frac{20n}{\eta^2 \ell} + 1 \le \frac{20\delta \alpha}{\eta^2} + 1 \le \eta\alpha/4$ (using $\ell \ge n/\delta \alpha$, $\eta \gg \delta$, and $\alpha$ being large), a contradiction to the minimality of $P_3$.
			To finish, note that there is an $xy$-path $P_1'$ in $G[V(P_1)]$ of length exactly $\ell - \ell(P_3)$. Indeed, since $\ell(P_1) = \ell/2$, we have $\ell(P_1) - \eta \alpha/4 \le \ell - \ell(P_3) \le \ell(P_1)$, and $P_1$ has at least $\eta \alpha/4$ non-intersecting chords of length $2$, so we can use some of these chords to shorten $P_1$ as needed. Concatenating $P_1'$ and $P_3$ yields a cycle of the desired length.
		\end{proof}

	\subsection{Finding a cycle with many chords}
		In this subsection we prove \Cref{lem:long}, the main step towards \Cref{lem:upper-range}. 
		The proof proceeds roughly as follows. We take $C$ to be a longest cycle which satisfies \ref{itm:long-jump} and \ref{itm:long-chords}. Our task is to show that \ref{itm:long-length-4} also holds, so suppose not. This means that $|C| < \ell$ and there is a component $H$ in $G - V(C)$ that contains a path of length at least $4$. To reach a contradiction, we will show that there is a cycle $C'$ satisfying: $|C'| > |C|$; the subgraph $C' \setminus V(C)$ is a subpath $P'$ contained in $H$; and $C'$ contains either $P_0$ or many chords of length $2$. For this to be a contradiction, we want $|P'| \le \frac{n}{\delta \alpha}$, due to \ref{itm:long-jump}. This will be achieved in the first two claims, \Cref{claim:med-path,claim:diam}, the latter asserting that if any two vertices $x,y$ in $H$ can be joined by a path of length at least $4$ in $H$, then there is a path in $H$ joining $x$ and $y$ whose length is at most $\frac{n}{\delta \alpha}$. It is easy to conclude that $|C| > \alpha$ (see \Cref{claim:C-length}). Now, a standard rotation technique yields a cycle $C'$ such that: $|C'| > |C|$; $C'$ misses at most two edges of $C$; and the vertices not in $C'$ induce a subpath of $C'$. This implies the desired contradiction if $C$ has at least $\delta \alpha + 2$ chords of length $2$, so we may assume this is not the case (see \Cref{claim:chords}). A more intricate rotation argument (which was outlined in the introduction) implies that $H$ has no path of length at least $4$ whose ends have relatively small degree in $H$ (see \Cref{claim:A-edges}). 
		We then use structural arguments to find a relatively large subset $A \subseteq V(H)$ whose every two vertices can be joined by a path of length at least $4$ in $H[A]$, and whose almost every vertex has few neighbours in $H - A$. A similar rotation argument now yields the desired contradiction, using the properties of $A$.

		\begin{proof}[Proof of \Cref{lem:long}]
			Let $\eta$ be a constant satisfying $\delta \ll \eta \ll 1$.
			Let $C$ be a longest cycle satisfying \ref{itm:long-jump} and \ref{itm:long-chords}, with the additional technical condition that if $P_0 \not\subseteq C$ then $|C| > \alpha$ (which later allows us to assume that $|C| > \alpha$, a useful property for an argument based on rotations). By assumption, there is a cycle of length at most $\ell$ that extends $P_0$, showing that such a cycle $C$ exists. We will show that $C$ also satisfies \ref{itm:long-length-4}, which would prove the lemma. Suppose not, then $|C| < \ell$ and $G - V(C)$ contains a path of length $4$. Let $H$ be a component in $G - V(C)$ that contains such a path.

			Define $E$ as follows. If $P_0 \subseteq C$, set $E = E(P_0)$. Otherwise, let $E$ be a set of $2\delta \alpha$ edges in $C$ corresponding to $\delta \alpha$ pairwise non-intersecting chords of length $2$ in $C$.

			The next claim is in preparation for the claim after that, which says that any two vertices in $H$ which are joined by a path of length at least $4$ in $H$, are joined by such a path which is not too long. This will help us control the length by which the cycle $C$ is extended in various situations, allowing us to satisfy \ref{itm:long-jump}.

			\begin{claim} \label{claim:med-path}
				There is no path $P$ in $H$ with $\frac{n}{\eta^3\alpha} \le |P| \le \frac{n}{\delta\alpha}$ whose ends have degree at most $(\frac{1}{2} - \eta) \alpha$ in $H$.
			\end{claim}

			\begin{proof}
				Suppose that $P$ is such a path and denote its ends by $x, y$.
				Let $X$ be the set of neighbours of $x$ in $C$ that are not incident to edges of $E$; then $|X| \ge d_G(x) - d_H(x) - 2|E| \ge \alpha - (\frac{1}{2} - \eta)\alpha - 4\delta \alpha \ge (\frac{1}{2} + \eta^2)\alpha$. Define $Y$ similarly with respect to $y$, then similarly $|Y| \ge (\frac{1}{2} + \eta^2)\alpha$.

				Write $\Xm = \{\xm : x \in X\}$ and define $\Ym$ analogously for $y$.
				We claim that $\Xm \cup \Ym$ is independent.
				Indeed, suppose not, and pick $u,v \in X \cup Y$ such that $\um \vm$ is an edge.
				If $u, v \in X$, define 
				\begin{equation} \label{eqn:C-1}
					C' = (u C_{u \to \vm} \vm \um C_{\um \nto v} v x u).
				\end{equation}
				(See \Cref{fig:cycle1}, and note that this make sense also when $\um = v$.)
				\begin{figure}[h]
					\centering
					\begin{subfigure}[b]{.4\textwidth}
						\centering
						\includegraphics[scale = .7]{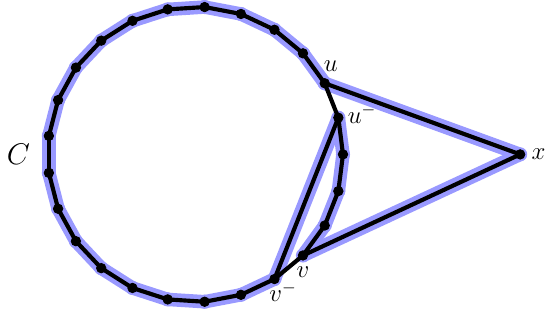}
						\caption{A cycle as in \eqref{eqn:C-1}}
						\label{fig:cycle1a}
					\end{subfigure}
					\hspace{.5cm}
					\begin{subfigure}[b]{.4\textwidth}
						\centering
						\includegraphics[scale = .7]{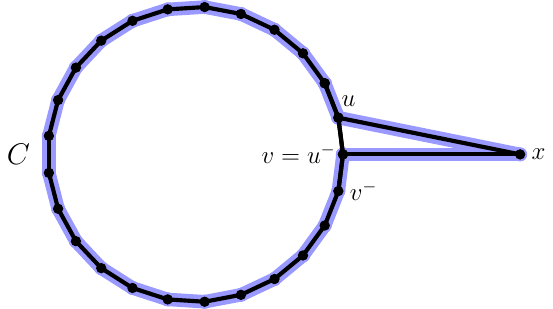}
						\caption{A cycle as in \eqref{eqn:C-1} when $\um = v$}
						\label{fig:cycle1b}
					\end{subfigure}
					\caption{Cycles as in \eqref{eqn:C-1}}
					\label{fig:cycle1}
				\end{figure}
				Moreover, $|C'| = |C| + 1 \le \ell$ and $E \subseteq C'$, so \ref{itm:long-jump} and \ref{itm:long-chords} both hold, a contradiction to the maximality of $C$. The same argument works if $u,v \in Y$, so suppose now that $v \in X$ and $u \in Y$, and define
				\begin{equation} \label{eqn:C-2}
					C' = (u C_{u \to \vm} \vm \um C_{\um \nto v} v x P y u).
				\end{equation}
				(See \Cref{fig:cycle2}, and note that this also works when $v = \um$.)
				\begin{figure}[h]
					\centering
					\includegraphics[scale = .7]{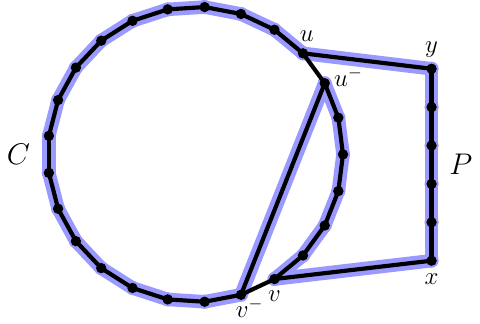}
					\vspace{-.5cm}
					\caption{A cycle as in \eqref{eqn:C-2}}
					\label{fig:cycle2}
				\end{figure}
				
				Then $|C'| \le |C| + |P| \le |C| + \frac{n}{\delta \alpha}$, contradicting the maximality of $C$. 

				Since $\Xm \cup \Ym$ is independent, it has size at most $\alpha$, showing that $|\Xm \cap \Ym| \ge \eta^2 \alpha$, and thus $|X \cap Y| \ge \eta^2 \alpha$. Consider the segments of $C$ between consecutive elements of $X \cap Y$. At most $2 \delta \alpha$ of them contain an edge of $E$, leaving at least $\eta^3 \alpha$ segments without any edges of $E$. Pick a segment $I$ of length at most $\frac{|C|}{\eta^3\alpha} \le \frac{n}{\eta^3\alpha}$, and denote its ends by $u$ and $v$. Define 
				\begin{equation} \label{eqn:C-3}
					C' = (vC_{v \nto u}uyPxv),
				\end{equation}
				(see \Cref{fig:cycle3}). Then $|C'| \ge |C| + |P| - (|I| - 2) \ge |C| + 1 > |C|$ (using $|P| \ge \frac{n}{\eta^3\alpha} \ge \ell(I) = |I| - 1$) and $|C'| \le |C| + |P| \le |C| + \frac{n}{\delta \alpha}$. This contradicts the maximality of $C$. 
				\begin{figure}[h]
					\centering
					\includegraphics[scale = .7]{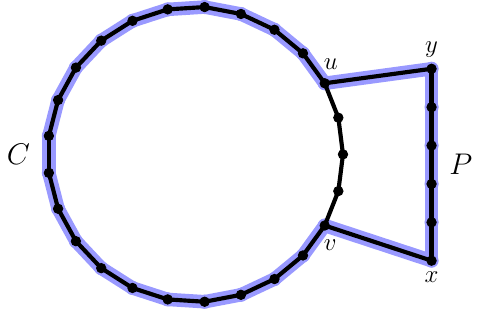}
					\caption{A cycle as in \eqref{eqn:C-3}}
					\label{fig:cycle3}
				\end{figure}
			\end{proof}

			\begin{claim} \label{claim:diam}
				If $x,y$ are vertices in $H$ that are joined by a path of length at least $4$, then there is such a path $P$ in $H$ with $|P| \le \frac{n}{\delta \alpha}$.
			\end{claim}

			\begin{proof}
				Let $x, y \in V(H)$ which can be joined by a path of length at least $4$ in $H$. Let $P$ be a shortest $xy$-path in $H$ of length at least $4$.
				We need to show $|P| \le \frac{n}{\delta \alpha}$. Suppose this is not the case. 

				Denote the set of first and last four vertices in $P$ by $X$, let $U$ be a maximal set of vertices in $P - X$, that are pairwise at distance at least $3$ from each other on $P$; so $|U| \ge \floor{\frac{|P|-8}{3}} \ge \frac{|P|}{4} \ge \frac{n}{5\delta \alpha}$. Let $W$ be the set of vertices in $U$ with degree at least $\alpha/4$ in $H$. Then the sets $N_H(w) - X$, with $w \in W$, are pairwise disjoint sets (due to the minimality of $P$) of size at least $\alpha/4 - 8 \ge \alpha / 5$. It follows that $|W| \le \frac{n}{\alpha / 5} \le |U|/2$.

				By the previous paragraph, at least $|U|/2 \ge |P|/8$ of the vertices in $P$ have degree at most $\alpha/4$ in $H$. 
				Since $P$ can be divided into at most $\frac{|P|}{n/2\delta \alpha}$ paths on at most $\frac{n}{\delta \alpha}$ vertices (by taking as many paths as possible of order exactly $\frac{n}{2\delta \alpha}$ and adding the remainder to the last path), there exists a subpath $Q$ of order at most $\frac{n}{\delta \alpha}$ that contains at least $\frac{|P|/8}{|P|/(n/2\delta\alpha)} = \frac{n}{16\delta\alpha}$ vertices of degree at most $\alpha/4$ in $H$. 
				Then $Q$ in turn contains a subpath of order between $\frac{n}{16\delta\alpha}$ and $\frac{n}{\delta\alpha}$ whose ends have degree at most $\alpha/4 \le (\frac{1}{2} - \eta)\alpha$ in $H$, contradicting \Cref{claim:med-path}.
			\end{proof}

			For technical reasons, it will be useful to know that $|C| > \alpha$. This is proved in the following claim.

			\begin{claim} \label{claim:C-length}
				$|C| > \alpha$.
			\end{claim}

			\begin{proof}
				Suppose not. Then, by connectivity, every vertex in $C$ has a neighbour in $H$.
				Notice that $C$ has two consecutive vertices $u,v$ that are not the ends of the same edge in $E$; this is because, by $|C| \le \alpha$ we have $P_0 \subseteq C$ and so $E$ is the edge set of $P_0$.
				Let $x$ be a neighbour of $v$ in $H$, let $y$ be a neighbour of $u$ in $H$, and let $P$ be a shortest $xy$-path in $H$ (if $x = y$ then $P$ is a single vertex). Then $|P| \le \frac{n}{\delta\alpha}$ by \Cref{claim:diam}. A cycle as in \eqref{eqn:C-3} yields a contradiction to the maximality of $C$.
			\end{proof}

			\begin{claim} \label{claim:chords}
				There is no collection of at least $\delta \alpha + 2$ non-intersecting chords in $C$ of length $2$.
			\end{claim}

			\begin{proof}
				Suppose that $C$ has at least $\delta \alpha + 2$ non-intersecting chords of length $2$.
				Let $U$ be the set of vertices in $C$ with neighbours in $H$.
				Then, because $\kappa(G), |C| > \alpha$ (by assumption and \Cref{claim:C-length}), we have $|U| > \alpha$, implying that there are $u,v \in U$ such that $\um \vm$ is an edge. 
				Let $x,y$ be neighbours of $v, u$ in $H$, and let $P$ be a shortest path in $H$ between $u$ and $v$. Define $C'$ as in \eqref{eqn:C-2} (see \Cref{fig:cycle2}). Then $|C'| > |C| > \alpha$ and $|C'| \le |C| + |P| \le \ell + \frac{n}{\delta \alpha}$, by \Cref{claim:diam}. Moreover, since $C'$ misses at most two edges of $C$, it has at least $\delta \alpha$ non-intersecting chords of length $2$. This contradicts the maximality of $C$.
			\end{proof}

			The following claim contains the most crucial novelty in the proof.

			\begin{claim} \label{claim:length-4}
				There is no path of length at least $4$ in $H$ whose ends have degree at most $(\frac{1}{2} - \eta)\alpha$ in $H$.
			\end{claim}

			\begin{proof}
				Suppose that $P$ is a path of length at least $4$ whose ends, denoted $x$ and $y$, have degree at most $(\frac{1}{2} - \eta)\alpha$ in $H$. By \Cref{claim:diam}, we may assume $|P| \le \frac{n}{\delta\alpha}$.

				Let $X$ be the set of vertices in $C$ that are neighbours of $x$ but not of $y$, and are at distance at least $4$ on $C$ from any edge in $E$. Define $Y$ similarly with the roles of $x$ and $y$ reversed, and let $Z$ be the set of vertices in $C$ that are neighbours of both $x$ and $y$, and are at distance at least $4$ from any edge in $E$. Then $|X \cup Z| \ge \alpha - d_H(x) - 10|E| \ge (\frac{1}{2} + \eta^2)\alpha$, with the same estimate holding for $|Y \cup Z|$. Define 
				\begin{equation*}
					\Xm = \{\xm : x \in X\} \qquad
					\Ym = \{\ym : y \in Y\} \qquad
					\Zm = \{\zm : z \in Z\} \qquad
					Z\mthree = \{z\mthree : z \in Z\}. 
				\end{equation*}
				We claim that these four sets are pairwise disjoint. By disjointness of $X, Y, Z$, if not, then $Z\mthree \cap (X \cup Y \cup Z) \neq \emptyset$, showing that there are vertices $u, v$ at distance $2$ on $C$ such that $v$ is a neighbour of $x$ and $u$ a neighbour of $y$.
				This is a contradiction, as can be seen by defining $C'$ as in \eqref{eqn:C-3} (see \Cref{fig:cycle3}).
				Indeed, then $|C'| = |C| - 1 + |P| > |C|$ (using $|P| \ge 5$), and $|C'| \le |C| + |P| \le |C| + \frac{n}{\delta \alpha}$.
				So, the sets $\Xm, \Ym, \Zm, Z\mthree$ are pairwise disjoint, showing that $|\Xm \cup \Ym \cup \Zm \cup Z\mthree| = |X \cup Z| + |Y \cup Z| \ge (1 + 2\eta^2)\alpha$.
				Since $\alpha(G) = \alpha$, there is a matching $M$ of size at least $\eta^2 \alpha$ in $\Xm \cup \Ym \cup \Zm \cup Z\mthree$. 

				Let $e$ be an edge in $M$. We claim that it is of form $\zm z\mthree$, for some $z \in Z$.
				Indeed, suppose not. Then one of the following holds: there exist $u,v \in X \cup Y \cup Z$ such that $e = \um\vm$; there exist $u \in X \cup Y \cup Z$ and $v \in Z$ such that $e = \um v\mthree$; and there exist $u,v \in Z$ such that $e = u\mthree v\mthree$.

				We show now that each of these cases leads to a contradiction. 
				This can be seen in the first case via a cycle as in \eqref{eqn:C-1} or \eqref{eqn:C-2}.
				In the second case, by symmetry, we may assume that $u$ is a neighbour of $y$. Consider the cycle
				\begin{equation} \label{eqn:C-4}
					C' = (uC_{u \to v\mthree} v\mthree \um C_{\um \nto v} vxPyu). 
				\end{equation}
				(See \Cref{fig:cycle4}.)
				Then, $|C'| = |C| - 2 + |P| > |C'|$, and $|C'| \le |C| + |P| \le |C| + \frac{n}{\delta \alpha}$.
				(Here we are implicitly assuming that $u, v$ are at distance at least $4$ on $C$; if this is not the case then we can reach a contradiction via a cycle of form \eqref{eqn:C-3}.)
				Finally, in the third case we can similarly reach a contradiction via the cycle
				\begin{equation} \label{eqn:C-5}
					C' = (uC_{u \to v\mthree} v\mthree u\mthree C_{u\mthree \nto v} vxPyu). 
				\end{equation}
				(See \Cref{fig:cycle5}.)
				\begin{figure}[h]
					\centering
					\begin{subfigure}[b]{.4\textwidth}
						\centering
						\includegraphics[scale = .7]{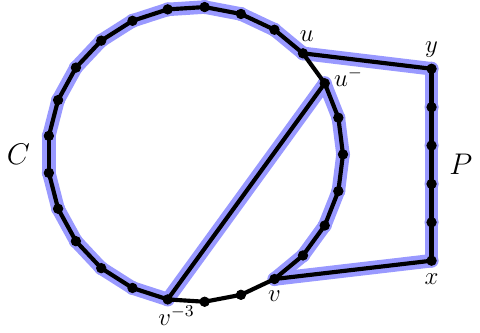}
						\caption{A cycle as in \eqref{eqn:C-4}}
						\label{fig:cycle4}
					\end{subfigure}
					\hspace{.5cm}
					\begin{subfigure}[b]{.4\textwidth}
						\centering
						\includegraphics[scale = .7]{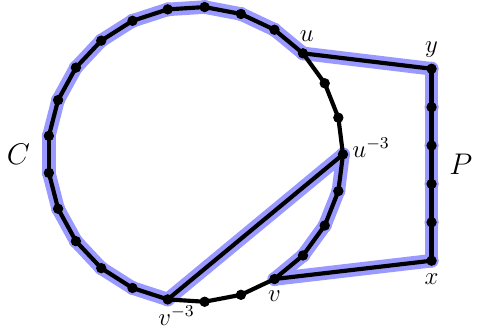}
						\caption{A cycle as in \eqref{eqn:C-5}}
						\label{fig:cycle5}
					\end{subfigure}
					\caption{Cycles as in \eqref{eqn:C-4} and \eqref{eqn:C-5}}
				\end{figure}

				The previous paragraph implies that every edge in $M$ is of the form $\zm z\mthree$ for some $z \in Z$. In particular, $M$ is a set of $\eta^2 \alpha \ge \delta \alpha + 2$ pairwise non-intersecting chords of length $2$ in $C$, contradicting \Cref{claim:chords}.
			\end{proof}

			Consider a longest path $P$ in $H$. Then $\ell(P) \ge 4$, by assumption on $H$. By the last claim, at least one of $P$'s ends has degree at least $(\frac{1}{2} - \eta)\alpha$ in $H$; say $x$ is such an end. Then all of $x$'s neighbours in $H$ are in $V(P)$, by maximality of $P$, showing that there is a cycle $C'$ in $H$ of length at least $(\frac{1}{2} - \eta)\alpha$. Notice that any two vertices in $C'$ can be joined by a path of length at least $4$ in $C'$. Thus we can take $A$ to be a largest set satisfying $V(C') \subseteq A \subseteq V(H)$, such that for any two vertices in $A$ there is a path of length at least $4$ in $H[A]$ that joins these two vertices.

			\begin{claim} \label{claim:A-edges}
				There is a subset $A' \subseteq A$ of size at least $|A| - \delta^{-3}$ such that every $a \in A'$ has at most $\delta \alpha$ neighbours in $V(H) - A$.
			\end{claim}

			\begin{proof}
				We claim that there is no path $P$ in $H$ of length at least $2$, whose ends are in $A$ but whose interior is outside of $A$. Indeed, otherwise, take $A' = A \cup V(P)$. It is easy to check that every two vertices $x, y \in A'$ can be joined by a path of length at least $4$ in $H[A']$. Indeed, this holds by choice of $A$ if $x,y \in A$, so suppose that $x$ is in the interior of $P$. If $y \in A$, let $z$ be an end of $P$ which is not $y$. Then concatenating a $yz$-path of length at least $4$ in $H[A]$ with the segment $P_{z \to x}$, yields a path with the required properties. 
				Similarly, if $x,y$ are both in the interior of $P$, let $z, w$ be the ends of $P$ such that $z$ is closer to $x$ than to $y$ in $P$. Now concatenate $P_{x \to z}$ with a $zw$-path of length at least $4$ in $H[A]$, and with $P_{w \to y}$, to get a path as desired.

				Similarly, there is no cycle of length at least $8$ that intersects $A$ in exactly one vertex.

				By the first paragraph, for every $x \in V(H) - A$ there is a vertex $a(x) \in A$ such that $a(x)$ disconnects $x$ from $A$ in $H$. Let $X(a)$ be the set of vertices $x \in V(H) - A$ such that $a(x) = a$.
				Notice that, by \Cref{claim:length-4} and the choice of $A$, for at most one $a \in A$, the set $X(a)$ contains a vertex of degree less than $(\frac{1}{2} - \eta)\alpha$ in $H$. Let $a'$ be this vertex (if exists).

				Let $\cC(a)$ be a maximal collection of pairwise vertex-disjoint cycles of length at least $\delta \alpha$ in $H[X(a) - \{a\}]$, for $a \in A - \{a'\}$. Let $A'$ be the set of vertices $a \in A - \{a'\}$ such that $|\cC(a)| \le \delta \alpha$. Then $|A - A'| \le \frac{n}{(\delta \alpha)^2} + 1 \le \frac{4}{\delta^2} + 1 \le \delta^{-3}$, by the assumption $n \le 4\alpha^2$, and by disjointness of the cycles in $\cC(a)$ and of the sets $X(a)$. 
				
				It remains to show that every $a \in A'$ has at most $\delta \alpha$ neighbours in $V(H) - A$.
				Indeed, let $a \in A'$, and write $\cC(a) = \{C_1, \ldots, C_r\}$.
				Notice that there are no two vertex-disjoint (except at $a$) paths from $a$ to $C_i$, for any $i \in [r]$, because otherwise the second paragraph would be violated. There is therefore a vertex $y_i$ that disconnects $a$ from $C_i$, for $i \in [r]$. If $a$ has more than $r$ neighbours in $V(H) - A$, let $x$ be such a neighbour which is not in $\{y_1, \ldots, y_r\}$. Then the connected component $H'$ of $H[X(a) - (\{a, y_1, \ldots, y_r\})]$ that contains $x$ is a non-empty graph with minimum degree at least $(\frac{1}{2} - \eta)\alpha - (r+1) \ge \delta \alpha$. So $H'$ has a cycle $C'$ of length at least $\delta \alpha$. Notice that $C'$ is vertex-disjoint of the cycles $C_1, \ldots, C_r$, contradicting the maximality of $\cC(a)$. This shows that $a$ has at most $r \le \delta \alpha$ neighbours in $H - A$, for every $a \in A'$.
			\end{proof}

			Let $W$ be the set of vertices in $C$ with at least two neighbours in $A$.
			We claim that $|A| + |W| \ge (\frac{1}{2} + \eta)\alpha$.
			Indeed, this is clearly the case if $|A| \ge (\frac{1}{2} + \eta)\alpha$, so suppose otherwise. Let $U$ be the set of vertices in $C$ that are neighbours of vertices in $A$, and let $U'$ be the set of vertices in $U$ which are not in edges in $E$. Notice that $(U')^-$ is independent, because otherwise a cycle as in \eqref{eqn:C-1} or \eqref{eqn:C-2}, together with \Cref{claim:diam}, would lead to a contradiction. Thus $|U'| \le \alpha$ and $|U| \le |U'| + 2|E| \le (1 + 2\delta)\alpha$. We conclude that $|W| \ge \alpha/3$, say. 
			Indeed, otherwise each vertex in $A'$ contributes at least $(\frac{1}{2} - \eta -\delta - \frac{1}{3})\alpha \ge \frac{\alpha}{7}$ distinct vertices to $U$, resulting in $U$ having size at least $|A'| \cdot \frac{\alpha}{7} \ge  \frac{\alpha^2}{28} > (1 + 2\delta)\alpha$, a contradiction. It follows that $|A| + |W| \ge (\frac{1}{2} - \eta + \frac{1}{3})\alpha \ge (\frac{1}{2} + \eta)\alpha$, as claimed (using $|A| \ge (\frac{1}{2} - \eta)\alpha$).

			By connectivity and by $|C| > \alpha$, there are at least $\min\{|A|, \alpha - |W|\}$ pairwise vertex-disjoint paths from $A$ to $C - W$; denote such a collection of paths by $\cP$ and assume that interiors of paths in $\cP$ are disjoint of $C$ and $A$. Let $X$ be the set of vertices in $C$ that are either in $W$ or an end of a path in $\cP$, and moreover are at distance at least $4$ from $E$. Then $|X| \ge \min\{|A| + |W|, \alpha\} - 8\delta \alpha \ge (\frac{1}{2} + \eta^2)\alpha$. 
			
			\begin{claim} \label{claim:W}
				Let $u, v \in X$ be distinct. Then there are distinct vertices $x,y$ in $H$ such that $x$ is a neighbour of $v$ and $y$ of $u$, and there is an $xy$-path of length at least $4$ in $H$.
			\end{claim}

			\begin{proof}
				By symmetry, there are three cases to consider: $u,v \in W$; $v \notin W$, $u \in W$; and $u,v \notin W$.
				In the first case, there are distinct $x,y$ in $A$ such that $x$ is a neighbour of $v$ and $y$ of $u$. By choice of $A$, there is an $xy$-path of length at least $4$ in $H$. In the second case, let $P$ be any path from $v$ to $A$, whose interior is disjoint of $C$ and $A$, and denote its end which is not $v$ by $x'$. Let $y \in A$ be a neighbour of $u$ which is not $x$, and let $Q$ be an $x'y$-path of length at least $4$ in $H[A]$. Taking $x$ to be the neighbour of $v$ in $P$, the path $PQ$ with $v$ removed is an $xy$-path of length at least $4$. Finally, in the third case, we may pick two vertex-disjoint paths $P_1,P_2$ from $v,u$ to $C$, with interior disjoint of $C$ and $A$. Let $x',y'$ be the ends of $P_1,P_2$ in $A$, respectively, let $x,y$ the neighbours of $v,u$ in $P_1,P_2$, and let $Q$ be an $x'y'$-path of length at least $4$ in $A$. Then $P_1QP_2 - \{u,v\}$ is an $xy$-path of length at least $4$. 
			\end{proof}

			One can now show that any two vertices in $X$ are at distance at least $4$ in $C$, via a cycle as in \eqref{eqn:C-3}, using \Cref{claim:W,claim:diam}.

			In particular, the sets $\Xm$ and $X\mthree$, defined as usual, are disjoint. Thus $|\Xm \cup X\mthree| \ge (1 + 2\eta^2)\alpha$, showing that there is a matching $M$ in $\Xm \cup X\mthree$ of size at least $\eta^2\alpha$. 

			We claim that all edges in $M$ are of form $\um u\mthree$, for some $u \in X$. Indeed, if not, there are distinct $u,v \in X$ such that one of $\um\vm$, $\um v\mthree$, $u\mthree, v\mthree$ is an edge in $M$. Each of these cases leads to a contradiction, by considering cycles as in \eqref{eqn:C-2}, \eqref{eqn:C-4}, \eqref{eqn:C-5}, using \Cref{claim:W,claim:diam} to find an appropriate $P$ with $5 \le |P| \le \frac{n}{\delta\alpha}$.

			Thus $M$ is a matching of at least $\eta^2\alpha$ non-intersecting chords of length $2$, a contradiction to \Cref{claim:chords}.
			This concludes the proof that $G - V(C)$ has no paths of length $4$, so $C$ satisfies the requirements of the lemma.
		\end{proof}

	\subsection{Extending a cycle whose complement is $P_5$-free}

		In this subsection we prove \Cref{lem:length-3-remainder}. We first sketch its proof. Suppose that $C$ is a cycle in a graph $G$ satisfying: $|C| < \ell$; $G - V(C)$ is $P_5$-free; and $C$ has $18$ non-intersecting chords of length $2$ or $3$. A simple rotation-extension argument (see \Cref{prop:cycle-extension}) implies that every cycle $C'$ containing the vertices of $C$ can be extended by at most four vertices, from any specified component of $G - V(C')$. This alone allows us to find a cycle of length between $\ell-3$ and $\ell$. Another consequence of \Cref{prop:cycle-extension} is that, if for such a cycle $C'$ there is a vertex outside of $C'$ with at most one neighbour outside of $C'$, then $C'$ can be extended by exactly that one vertex. It thus suffices to find a cycle $C'$ such that: $V(C) \subseteq V(C')$; $|C'| \le \ell$; and $G - V(C')$ has at least three vertices in components which are trees. This can be done by three applications of \Cref{prop:tree}, asserting that a cycle $C'$ containing the vertices of $C$ can be extended by at most four vertices from a specified component $H$ of $G - V(C')$, such that the remainder of the component is a non-empty tree. This might fail if $|C|$ is too close to $\ell$, in which case we can use the short chords in $C$ to get a cycle of length $\ell$.

		Before turning to the proof, we state and prove three propositions.
		The first one is a consequence of a simple rotation and extension operation. 

		\begin{proposition} \label{prop:cycle-extension}
			Let $G$ be a graph satisfying $\kappa(G) > \alpha(G)$.
			Let $C$ be a cycle, $H$ a component in $G - V(C)$, and $u,v$ distinct vertices in $H$.
			Then there is a cycle $C'$ such that: $V(C')$ is the union of $V(C)$ with the vertices of a path in $H$ that contains $u$ and avoids $v$; and $C'$ contains all but at most two edges of $C$. 
		\end{proposition}

		\begin{proof}
			Write $\alpha = \alpha(G)$. Because $\kappa(G) > \alpha$, the graph $G - \{v\}$ is $\alpha$-connected. 
			If there is a vertex $w$ in $C$ such that $w$ and $\wm$ can be joined by a path in $G - \{v\}$, whose interior contains $u$ and is disjoint of $C$, then we can define 
			\begin{equation*}
				C' = (wC_{w \to \wm}\wm P w).
			\end{equation*}
			Then $C'$ contains all edges in $C$ except for $w \wm$ and its vertex set is the union of $V(C)$ with the vertices of the path $P - \{w,  \wm\}$. Suppose now that no such vertex $w$ exists.

			If $|C| < \alpha$, then by connectivity there are $|C|$ pairwise vertex-disjoint (except at $u$) paths from $u$ to $C$ in $G - \{v\}$, leading to a contradiction to the assumption we just made. So $|C| \ge \alpha$, implying that there are pairwise vertex-disjoint (except at $u$) paths $P_1, \ldots, P_{\alpha}$ from $u$ to $C$ in $G - \{v\}$. We assume that the interiors of the paths $P_i$ are disjoint of $C$. Let $w_i$ be the end of $P_i$ which is not $u$. Then by the above assumption no two vertices in $\{w_1, \ldots, w_{\alpha}\}$ are consecutive, and there is no edge between $u$ and $\wm_i$ for any $i \in [\alpha]$.

			The set $\{\wm_i : i \in [\alpha]\} \cup \{u\}$ has size $\alpha + 1$, and so it spans an edge $e$. Since there is no edge of form $u\wm_i$, we have $e = \wm_i \wm_j$ for some distinct $i,j \in [\alpha]$. We then get the cycle
			\begin{equation*}
				C' = (u P_i w_i C_{w_i \to \wm_j} \wm_j \wm_i C_{\wm_i \nto w_j} w_j P_j u),
			\end{equation*}
			like in \eqref{eqn:C-2} or \Cref{fig:cycle2} but with different notation.
			It is easy to check that $C'$ satisfies the requirements of the proposition. Indeed, $C'$ uses all edges of $C$ except for $w_i\wm_i$ and $w_j \wm_j$, and $V(C')$ is the union of $V(C)$ with the vertices of the path $(P_i \cup P_j) - \{w_i, w_j\}$.
		\end{proof}

		The next proposition gives some structural information regarding $P_5$-free graphs.

		\begin{proposition} \label{prop:P-5-free-structure}
			Let $H$ be a connected $P_5$-free graph. Then one of the following holds: $H$ is a complete graph on four vertices; $H$ is a tree; or there is a vertex $u$ such that $H - \{u\}$ is a forest.
		\end{proposition}

		\begin{proof}
			Suppose that $H$ is not a tree. Then it contains a cycle $C$. By $P_5$-freeness, $C$ has length $3$ or $4$. If it has length $4$, then $H$ has exactly four vertices, as adding a pendant edge to a $4$-cycle results in the existence of a $P_5$. In this case, it is easy to check that $H$ is either complete, or there is a vertex whose removal from $H$ leaves a tree.
			Suppose now that $C$ is a triangle and that $G$ has no $4$-cycles. The latter assumption implies that no two vertices in $C$ have a common neighbour outside of $C$. Hence, $P_5$-freeness implies that there is at most one vertex in $C$, say $u$, that has neighbours outside of $C$, and that any neighbour of $u$ outside of $C$ has no neighbours other than $u$. In particular, $H - \{u\}$ is a forest.
		\end{proof}

		The final proposition is a consequence of the former two.

		\begin{proposition} \label{prop:tree}
			Let $G$ be a graph satisfying $\kappa(G) > \alpha(G)$.
			Let $C$ be a cycle in a graph $G$ and let $H$ be a (non-empty) component in $G - V(C)$. Suppose that $H$ is $P_5$-free. Then there is a cycle $C'$ such that: $|C| \le |C'| \le |C|+4$; $V(C) \subseteq V(C')$; $C'$ misses at most four edges of $C$; and $H - V(C')$ is a non-empty forest.
		\end{proposition}

		\begin{proof}
			If $H$ is a tree, there is nothing to prove (take $C' = C$). 
			If there is a vertex $u$ in $H$ such that $H - \{u\}$ is a forest, then pick any $v \in V(H) - \{u\}$, and let $C'$ be a cycle as guaranteed by \Cref{prop:cycle-extension}. Then $|C| < |C'| \le |C| + 4$ (the upper bound is by $P_5$-freeness), $H - V(C')$ is a forest that contains $v$, and $C$ avoids at most two edges of $C$.

			It remains to consider the case $H \cong K_4$. By \Cref{prop:cycle-extension}, there is a cycle $C''$ such that $V(C) \subsetneq V(C'') \subsetneq V(C) \cup V(H)$, and $C''$ avoids at most two edges of $C$. Let $H' = H - V(C'')$. If $H'$ consists of at most two vertices, we are done. Otherwise, $H'$ is a triangle. Then we can apply \Cref{prop:cycle-extension} again to obtain a cycle $C'$ such that $V(C'') \subsetneq V(C') \subsetneq V(C'') \cup V(H')$ and $C'$ misses at most two edges of $C''$, and thus at most four edges of $C$. Now $H - V(C'')$ has at least one and at most two vertices, and so it is a tree. 
		\end{proof}

		Finally, here is the proof of the lemma.

		\begin{proof}[Proof of \Cref{lem:length-3-remainder}]
			Write $\alpha = \alpha(G)$.

			First, we pick cycles $C_1, C_2, C_3$ as follows, where $F_i$ is the union of components of $G - V(C_i)$ which are trees, and $G_i = G - V(F_i)$. 
			If $|F_{i-1}| \ge i$ or $|C_{i-1}| \ge \ell - |F_{i-1}|$, define $C_i = C_{i-1}$. Otherwise, we can apply \Cref{prop:tree}, with the cycle $C_{i-1}$ and any component in $G_{i-1} - C_{i-1}$, to find a cycle $C_i$ such that: $|C_{i-1}| \le |C_i| \le |C_{i-1}| + 4$; $V(C_{i-1}) \subseteq V(C_i)$; $C_i$ misses at most four edges of $C_{i-1}$; and $G_{i-1} - V(C_i)$ has a component which is a tree. 

			Notice that $|C_i| \le \min\{|C_0| + 4i, \ell + 3\}$; $|F_i| \ge \min\{i, \ell - |C_i|\}$; and $C_i$ avoids at most $4i$ edges of $C_0$, for $i \in [3]$.

			We distinguish three cases.
			\begin{itemize}
				\item
					$|C_3| = \ell + 2$. By the assumption that $C_0$ has at least $18$ pairwise non-intersecting chords of length $2$ or $3$, and the fact that $C_3$ misses at most $12$ edges of $C_0$, the cycle $C_3$ has either a chord of length $3$, or two non-intersecting chords of length $2$. Both scenarios allow for shortening $C_3$ by exactly $2$, to obtain an $\ell$-cycle.
				\item
					$|C_3| \in \{\ell + 1, \ell + 3\}$. Notice that $|F_3| \ge 1$ (otherwise $C_3 = C_0$, a contradiction to $|C_0| < \ell$); let $u$ be a vertex of degree at most $1$ in $F_3$. By \Cref{prop:cycle-extension}, there is a cycle $C_4$ with $V(C_4) = V(C_3) \cup \{u\}$ which misses at most two edges of $C_3$, and thus at most $14$ edges of $C$. So $|C_4| \in \{\ell + 2, \ell + 4\}$ and $C_4$ has four non-intersecting chords of length $2$ or $3$. These can be used to form a cycle of length $\ell$.
				\item
					$|C_3| \le \ell$.
					Let $C_4$ be a longest cycle satisfying $V(C_3) \subseteq V(C_4) \subseteq V(G) - V(F_3)$ and $|C_4| \le \ell$.
					We claim that $|C_4| \ge \ell - 3$ or $V(C_4) = V(G) - V(F_3)$.  
					Indeed, suppose not. Then we can apply \Cref{prop:cycle-extension} with a vertex in $V(G) - (V(F_3) \cup V(C_4))$ to obtain a cycle $C_5$ with $V(C_4) \subseteq V(C_5) \subseteq V(G) - V(F_3)$ and $|C_5| \le |C_4| + 4 \le \ell$, contradicting the maximality of $C_4$.

					Write $r = \ell - |C_4|$.
					Recalling that $|F_3| \ge \min\{3, \ell - |C_3|\}$, we have $|F_3| \ge r$. As $F_3$ is a forest, there are distinct vertices $u_1, \ldots, u_r$ in $F_3$ such that $u_j$ is a vertex of degree at most $1$ in $F_3 - \{u_1, \ldots, u_{j-1}\}$. 
					Apply \Cref{prop:cycle-extension} $r$ times, with $u_j$ and its unique neighbour in $F_3 - \{u_1, \ldots, u_{j-1}\}$ (if exists) in the $j^{\text{th}}$ application. These applications result in a cycle $C'$ with $V(C') = V(C) \cup \{u_1, \ldots, u_j\}$. In particular $|C'| = \ell$.
			\end{itemize}
			This completes the proof that $G$ contains a cycle of length $\ell$.
		\end{proof}

\section{Middle range} \label{sec:middle}

	In this section we prove \Cref{lem:middle-range} (restated here), about cycles whose length is not too small and not too large. 

	\lemMiddleRange*

	The general outline of the proof is similar to the proof of the middle range in \cite{draganic2023chvatal}. We start with a relatively short path $P$, which has the property that it can be shortened by any small amount (via short non-intersecting chords). We then extend $P$ to a cycle $C$ of length at most $n/\alpha$. In \cite{draganic2023chvatal} it was essentially trivial to show that such a cycle $C$ exists. Here we need to work quite a lot harder; see \Cref{lem:n-over-alpha}. 

	Next, writing $P' = C - E(P)$, we show that there is another path $P''$ with the same ends as $P'$, such that together with $P$ it forms a cycle of length slightly more than $\ell$. 
	This is achieved using \Cref{lem:mid-range}, which strengthens Lemma 2.10 in \cite{draganic2023chvatal}.

	Finally, this can be modified to a cycle of length exactly $\ell$, using the property of $P$. 

	\subsection{Lemmas}

		We shall need the following lemma, that builds on \Cref{lem:path-chords}, and yields a path with many chords that can be extended to a cycle of length at most $n/\alpha$.

		\begin{lemma} \label{lem:n-over-alpha}
			Let $0 < \delta \ll 1$, and let $\alpha, n$ be positive integers satisfying: $\alpha$ and $n/\alpha$ are large; $n \le 4\alpha^2$; and $n \ge \alpha$.
			Suppose that $G$ is a graph on $n$ vertices, with $\alpha(G) = \alpha$ and $\kappa(G) > \alpha$.
			Then there are paths $P_0$ and $P_1$, which have the same ends and vertex-disjoint interiors, such that: $\ell(P_0) \le 7\delta n / \alpha$; $\ell(P_1) \ge \delta n / \alpha$; $\ell(P_0) + \ell(P_1) \le n/\alpha$; and $P_0$ has at least $\delta n / \alpha$ non-intersecting chords of length $2$ or $3$, with at least one of length $2$. 
		\end{lemma}

		We will also use the following lemma, about extending a given path by a relatively small amount.

		\begin{lemma}\label{lem:mid-range}
			Let $\alpha, r, \ell, n$ be positive integers satisfying: $\alpha$ is large; $\max\{4\sqrt{\ell}, 32 \sqrt{\alpha}\} < r \le \min\{2\ell, n/\alpha, \alpha\}$; and $\ell \le \frac{n}{2}$.
			Suppose that $G$ is a graph on $n$ vertices, satisfying $\alpha(G) \le \alpha$ and $\kappa(G) \ge \alpha/2$, and that $P$ is a path of length $\ell$ in $G$. Then there is a path $P'$ with the same ends as $P$ that satisfies $|P| < |P'| \le |P| + r$.
		\end{lemma}

		We first show how to prove \Cref{lem:middle-range} using these lemmas, and then we prove both lemmas.

	\subsection{Proof of middle range lemma}

		Using the above lemmas, the proof of \Cref{lem:middle-range} is straightforward, albeit somewhat technical.

		\begin{proof}[Proof of \Cref{lem:middle-range}]
			Let $\eta$ be a constant satisfying $\delta \ll \eta \ll 1$.
			Apply \Cref{lem:n-over-alpha} with $\eta$ to find paths $P_0$ and $P_1$ with the same ends, denoted $x,y$, and disjoint interiors, such that: $\ell(P_0) \le 7\eta n/\alpha$; $\ell(P_1) \ge \eta n / \alpha$; $\ell(P_0) + \ell(P_1) \le n/\alpha$; and $P_0$ has $\eta n / \alpha$ non-intersecting chords of length $2$ or $3$, with at least one of length $2$.
			Consider the graph $G_1$, obtained from $G$ by removing the interior vertices of $P_0$.
			Then $\alpha(G_1) \le \alpha$, $|G_1| \ge n/2$, and, since $\ell(P_0) \le 7\eta n / \alpha \le \alpha/2$ (using $\alpha \ge \sqrt{n}/2$), we also have $\kappa(G) \ge \alpha/2$. Write $\ell' = \ell - \ell(P_0)$ and $r = \eta n / 2\alpha$.
			Let $P_2$ be a longest $xy$-path in $G_1$ of length at most $\ell' + r$. 

			\begin{claim}
				$\ell(P_2) \ge \ell'$.
			\end{claim}

			\begin{proof}
				Suppose not.
				Then apply \Cref{lem:mid-range} with $\alpha$, $r$, $\ell_{\ref{lem:mid-range}} = \ell(P_2)$ and $n_{\ref{lem:mid-range}} = |G_1|$. To see that the lemma is applicable, notice that, by assumption, $\ell(P_2) \le \ell' \le \ell \le n/4 \le |G_1|/2$, and $\kappa(G_1) \ge \alpha/2$.
				Moreover, we now verify that $\max\{4\sqrt{\ell''}, 32\sqrt{\alpha}\} \le r \le \min\{2\ell'', |G_1|/\alpha, \alpha\}$, where $\ell'' = \ell(P_2)$:
				\begin{itemize}
					\item
						$4\sqrt{\ell''} \le 4 \sqrt{\ell} \le 4 \sqrt{\delta} n/\alpha \le \eta n / 2\alpha = r$, using the assumption $\ell'' \le \ell'$ and that $\ell' \le \ell$.
					\item
						$\frac{32 \sqrt{\alpha}}{r} = \frac{64\alpha^{3/2}}{\eta n} \le \frac{64 \delta^{3/2}}{\eta} \le 1$, using $\alpha \le \delta n^{2/3}$, and showing $32\sqrt{\alpha} \le r$.
					\item
						$2\ell'' \ge 2\ell(P_1) \ge 2\eta n /\alpha \ge r$, by maximality of $P_2$.
					\item
						$\frac{|G_1|}{\alpha} \ge \frac{n}{2\alpha} \ge \frac{\eta n}{2\alpha} = r$.
					\item
						$\frac{\alpha}{r} = \frac{2\alpha^2}{\eta n} \ge \frac{1}{2\eta} \ge 1$, using $\alpha \ge \sqrt{n}/2$, and showing $\alpha \ge r$.
				\end{itemize}
				This shows that the lemma is indeed applicable. It implies the existence of an $xy$-path $P_2'$ with $\ell(P_2) < \ell(P_2') \le \ell(P_2) + r \le \ell' + r$, contradicting the maximality of $P_2$.
			\end{proof}

			Note that $\ell - \ell(P_2) \le \ell - \ell' = \ell(P_0)$ (by the claim) and $\ell - \ell(P_2) \ge \ell - \ell' - r = \ell(P_0) - \eta n / 2\alpha$ (by choice of $P_2$).
			Recalling that $P_0$ has $\eta n / \alpha$ non-intersecting chords of length $2$ or $3$, with at least one of length $2$, by \Cref{prop:chords} there is an $xy$-path $P_0'$ of length $\ell - \ell(P_2)$ with $V(P_0') \subseteq V(P_0)$. The concatenation of $P_0'$ and $P_2$ is a cycle of length exactly $\ell$.
		\end{proof}

	\subsection{Finding a short cycle with many short chords}
		\def \Pi {P_{\mathrm{init}}}

		We now prove \Cref{lem:n-over-alpha}. We start by taking a path $\Pi$ which is relatively short yet not too short, and has many chords of length $2$ or $3$, with at least one of length $2$; denote its ends by $x,z$.
		Now we consider the graph $G'$ obtained by removing the interior of $\Pi$, and observe that it has connectivity at least a little below $\alpha$. If the distance between $x$ and $z$ in $G'$ is sufficiently smaller than $n/\alpha$, we are done. If not, then due to space considerations, for almost all $i \le \alpha$, there are roughly $n/\alpha$ vertices at distance exactly $i$ from $x$ in $G'$. In particular, there is a relatively small $i$ such that the number of vertices at distance $j$ from $x$ in $G'$ is indeed roughly $n/\alpha$, for $j \in [i, i+6]$. We then focus on the vertices at distance between $i$ and $i+6$ from $x$ in $G'$, and repeat arguments similar to the ones used in \Cref{lem:path-chords}, to find the desired paths.

		\begin{proof}[Proof of \Cref{lem:n-over-alpha}]

			By \Cref{lem:path-chords}, there is a path $Q_1$ of length at most $\frac{3\delta n}{\alpha}$ with at least $\frac{\delta n}{\alpha}$ non-intersecting chords of length $2$ or $3$, at least one of which having length $2$. Denote the ends of $Q_1$ by $x, y$. By the connectivity assumption, there is a path $Q_2$ of length $\frac{\delta n}{\alpha}$, one of whose ends is $y$, and which otherwise does not intersect $Q_1$. Let $z$ be the other end of $Q_2$. Write $\Pi = x Q_1 y Q_2 z$; so $\ell(\Pi) \le \frac{4\delta n}{\alpha} \le 16 \delta \alpha$, using $n \le 4\alpha^2$. 

			Consider the graph $G'$, obtained from $G$ by removing the interior vertices of $\Pi$. Write 
			\begin{equation*}
				\kappa = (1 - 16\delta)\alpha,
				\qquad
				\qquad
				\ell = (1 - 4\delta)\cdot\frac{n}{\alpha},
			\end{equation*}
			and notice that $\kappa(G') \ge \alpha - |\Pi| \ge \kappa$. If $G'$ has an $xz$-path $Q_3$ of length at most $\ell$, we are done (take $P_0 = Q_1$ and $P_1 = yQ_2zQ_3x$), so suppose it does not.

			Let $U_i$ be the set of vertices at distance exactly $i$ from $x$ in $G'$. Notice that $U_i$ is non-empty for $i \le \ell$. By connectivity of $G'$, we have $|U_i| \ge \kappa$ for $i \in [\ell-1]$. Moreover, there is a matching of size at least $\kappa$ between $U_i$ and $U_{i+1}$, for $i \in [\ell-2]$. We claim that there exists $i$ with $3 \le i \le \frac{\eta n}{\alpha}$ such that the sets $U_{i}, U_{i+1}, \ldots, U_{i+6}$ all have size at most $(1 + \eta)\alpha$. Indeed, otherwise, we get
			\begin{align*}
				\sum_{i \in [\ell]} |U_i| 
				& \ge \kappa \left(\ell - \frac{\eta n}{\alpha}\right) + \left(\frac{\eta n}{7\alpha} - 1\right)\cdot \left(6 \kappa + (1 + \eta)\alpha\right) \\
				& \ge \kappa \ell - \frac{\kappa \eta n}{7 \alpha} + \frac{\eta n}{7} \cdot (1 + \eta) - 7(1 + \eta)\alpha \\
				& = \kappa \ell + \frac{\eta n}{7}\cdot \big(1 + \eta - (1 - 16\delta)\big) - 7(1 + \eta)\alpha \\
				& \ge (1 - 16\delta)(1 - 4\delta)n + \frac{\eta^2n}{7} - 8\alpha
				\ge \left(1 - 20\delta + \frac{\eta^2}{7} - \frac{8}{n/\alpha}\right)n > n,
			\end{align*}
			a contradiction. 
			For the first estimate we lower bounded $|U_i|$ by $\kappa$ for $i > \eta n / \alpha$, and divided the range $[3, \eta n / \alpha]$ into at least $\eta n / 7 \alpha - 1$ intervals of length $7$, and then lower bounded the total length of $|U_i|$ over each of these intervals by $6\kappa + (1 + \eta)\alpha$, which was possible due to the assumption that in each such interval there is a set $U_i$ of size at least $(1 + \eta)\alpha$.
			For the second inequality we used $\kappa \le \alpha$, and for the last estimate we used that $n/\alpha$ is large. 

			Fix an $i$ as above. Let $M_j$ be a matching of size $\kappa$ between $U_{i+j-1}$ and $U_{i+j}$, for each $j \in [6]$, and denote the union of these matchings by $F$. Let $U_{i+j}'$ be the set of vertices in $U_{i+j}$ that are in a path of length $6$ in this union. We claim that $|U_{i+j}'| \ge (1 - 13\eta)\alpha$. To see this, consider the set of vertices $U_{i+k}''$ in $U_{i+k}$ which are in paths of length $k-1$ in $F$ that start in $U_{i}$. Then $|U_i''| \ge \kappa \ge (1 - \eta)\alpha$, and $|U_{i+j}''| \ge |U_{i+j-1}''| - (|U_{i+j}| - \kappa) \ge |U_{i+j-1}''| - 2\eta \alpha \ge (1 - (2j+1)\eta)\alpha$. Each vertex in $U_{i+6}''$ is in a path of length $6$. Thus $|U_{i+j}'| \ge |U_{i+6}''| \ge (1 - 13\eta)\alpha$, for $j \in [0,6]$. 

			Rename $W_j = U_{i+j}'$ for $j \in [0, 6]$, and write $W = W_0 \cup \ldots \cup W_6$. Notice that every vertex in $W_1 \cup \ldots \cup W_5$ has at least $\kappa - 14\eta\alpha \ge \alpha/2$ neighbours in $W$.
			For a vertex $w \in W_j$ let $\wp$ be the $F$-neighbour of $w$ in $W_{j+1}$, for $j \in [0,5]$, and let $\wm$ be the $F$-neighbour of $w$ in $W_{j-1}$, for $j \in [6]$. Moreover, let $P(w)$ be the unique path through $w$ from $W_0$ to $W_6$ in $F$.
			Write $\ell' = \frac{\delta n}{\alpha}$.
			\begin{claim}
				There is a sequence of paths $P_0, \ldots, P_{2\ell'}$ with the following properties.
				\begin{enumerate}[label = \rm(\alph*)]
					\item \label{itm:n-alpha-disj}
						The paths $P_i$ are pairwise vertex-disjoint, except that the start point of $P_{i}$ is the end point of $P_{i-1}$, for $i \in [2\ell']$.
					\item \label{itm:n-alpha-len}
						$|P_i| \le 8$.
					\item \label{itm:n-alpha-chord}
						$P_i$ has a chord of length $2$ or $3$ for $i \in [0,2\ell']$ ($2$ if $i = 0$).
					\item \label{itm:n-alpha-end}
						The end point of $P_i$ is in $W_1 \cup \ldots \cup W_5$.
					\item \label{itm:n-alpha-W}
						$V(P_i) \subseteq W$ for $i \in [2\ell']$. If $V(P_0)$ is not fully contained in $W$, then it starts with a vertex in $\Pi$, such that the subpath of $\Pi$ between $x$ and $P_0$'s start point does not contain other vertices of $P_0$.
				\end{enumerate}
			\end{claim}

			\begin{proof}
				First, define $P_0$ as follows.
				Let $u \in W_3$. By assumption on $G$, the vertex $u$ lies in a triangle $(uvw)$. 
				If one of $v$ and $w$ is in $\Pi$, then without loss of generality $v$ is in $\Pi$ and is closer to $x$ than $w$ in $\Pi$ (if $w$ is also in $\Pi$).
				Define $P_0 = vwu$, and notice that $P_0$ satisfies the above requirements.

				Now suppose that $P_0, \ldots, P_{i-1}$ are defined and satisfy the above properties, where $i \in [\ell']$. Denote the end point of $P_{i-1}$ by $u$; then $u \in W_1 \cup \ldots \cup W_5$ by \ref{itm:n-alpha-W}.
				Let $W'$ be the set of vertices in $W$ which do not lie in any path $P(v)$, with $v \in V(P_0) \cup \ldots \cup V(P_i)$. Then $|W'| \ge |W| - 8 \ell' \ge |W| - \frac{\alpha}{4}$.
				Let $v$ be a neighbour of $u$ in $W'$, let $P'(v)$ be the subpath of $P(v)$ that starts at $v$ and ends in $W_3$, and let $w$ be the end of $P'(v)$. As in the proof of \Cref{lem:path-chords}, there exist $w_1, w_2, w_3 \in W' - V(P(v))$ such that one of the following holds: $(w w_1 w_2)$ is a triangle; $(w_1 w_2 w_3)$ is a triangle and $ww_1$ is an edge; or $(w w_1 w_2 w_3)$ is a $4$-cycle.
				Let $Q$ be $ww_1w_2$ in the first case, and $ww_1w_2w_3$ in the other two cases, and define $P_{i+1} = u v P'(v) w Q$. 
				\begin{figure}[h]
					\centering
					\includegraphics[scale = 1]{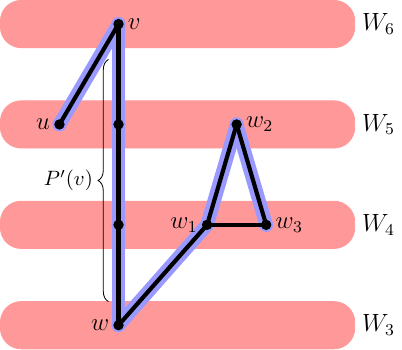}
					\caption{An illustration of how $P_i$ may look like}
					\label{fig:P-i}
				\end{figure}

				It is easy to check that $P_i$ satisfies the requirements. This completes the proof that paths $P_0, \ldots, P_{\ell'}$ as above exist.
			\end{proof}

			Define $P = P_0 \ldots P_{2\ell'}$, and denote the start point of $P$ by $u_0$ and the end point by $u_1$. 

			We define paths $Q_0$ and $Q_1$ between $x$ and $u_0$ and $u_1$.
			Let $i \in \{0,1\}$. If $u_i \in W$, let $v_i$ be a neighbour of $u_i$ such that $P(v_i)$ does not intersect $P(v)$ for any $v$ in $P$ (it is easy to see that such $v_i$ exists), and let $Q_i$ be the concatenation of a shortest path between $x$ and $v_i$ that does not intersect $P$ and the edge $v_iu_i$. Then $\ell(Q_i) \le \frac{2\eta n}{\alpha}$. If $u_i \notin W$, then $i = 0$ and $u_i \in V(\Pi)$, and we take $Q_0$ to be the subpath of $\Pi$ between $x$ and $u_0$. In this case we have $\ell(Q_0) \le \ell(\Pi) \le \frac{4\delta n}{\alpha}$.

			Thinking of $Q_0$ as a path from $x$ to $u_0$, let $x'$ be the last vertex in $Q_0$ which is also in $Q_1$ (notice that $x$ is in both $Q_0$ and $Q_1$, so this is well defined). Let $Q_0'$ be the subpath of $Q_0$ from $x'$ to $u_0$, and let $Q_1'$ be the subpath of $Q_1$ from $u_1$ to $x'$.
			The paths $P_0\ldots P_{\ell'-1}$ (for $P_0$) and $P_{\ell'}\ldots P_{2\ell'} Q_1 Q_0$ (for $P_1$) satisfy the requirements of the lemma.
			Indeed, the former has length at most $7\delta n/\alpha$ and at least $\delta n/\alpha$ non-intersecting chords of length $2$ or $3$, with at least one of length $2$, the latter has length at least $\delta n/\alpha$, and the total length is at most $(14\delta + 4\eta)n/\alpha \le n/\alpha$.
		\end{proof}
	\subsection{Extending middle range cycles}

		In this section we prove \Cref{lem:mid-range}. Here the task is to extend a given path $P$ by a small but positive amount.
		To do so, we use results about cycle-complete Ramsey numbers to find a cycle $C$ of specific length which is disjoint of $P$. We then use connectivity to find many pairwise vertex-disjoint paths between $P$ and $C$, one for each vertex of $C$. 
		If many of the paths are short, we can pick two that are close in $P$, and extend $P$ by going along these paths and along the longer subpath of $C$ that connects the ends of these paths in $P$ (see \Cref{fig:P'-case1}). Otherwise, we use the independence number to find an edge between two of the paths to get an appropriate extension (see \Cref{fig:P'-case2}).
		The proof here is reminiscent of the proof of Lemma 2.10 in \cite{draganic2023chvatal}, especially in the latter case, but the idea to use a cycle outside of $P$ is new and crucial for covering the full middle range.

		\begin{proof}[Proof of \Cref{lem:mid-range}]
			By \Cref{thm:ramsey-keevash}, the bound $r \ge 32 \sqrt{\alpha}$, and the assumption that $\alpha$ is large, we have the following bound on the cycle-complete Ramsey number: $r(C_{r/2}, K_{\alpha+1}) \le (r/2) \cdot \alpha \le n/2 \le n-\ell$. Since $\alpha(G) \le \alpha$, it follows that there is a cycle $C = (v_1 \ldots v_{r/2})$ of length $r/2$ which is vertex-disjoint of $P$.
			Since $\kappa(G) \ge \alpha/2 \ge r/2 = |C|$ and $|P| = \ell+1 \ge r/2$, there are $r/2$ pairwise vertex-disjoint paths from $C$ to $P$. Given such a collection of paths, denote the path starting with $v_i$ by $Q_i$, for $i \in [r/2]$, and let $u_i$ be the end of $Q_i$ in $P$. By assuming that the $Q_i$'s are as short as possible, the paths $Q_i$ are induced.
			We consider two cases: at least $r/4$ paths $P_i$ have length at most $r/4$; and at least $r/4$ paths have length at least $r/4$.

			Suppose that the former holds, and let $I$ be the set of indices $i \in [r/2]$ such that $\ell(Q_i) \le r/4$. Since $|I| \ge r/4$, there are distinct $i,j \in I$ such that the distance between $u_i$ and $u_j$ in $P$ is at most $\frac{\ell}{r/4} \le r/4$ (using $r \ge 4\sqrt{\ell}$). Let $P_1$ and $P_2$ be the two subpaths of $P$ obtained by removing the segment between $u_i$ and $u_j$, such that $u_i \in V(P_1)$ and $u_j \in V(P_2)$. Let $P_3$ be the longer subpath of $C$ with ends $v_i$ and $v_j$; so $\ell(P_3) \ge r/4$.
			Define $P'$ as follows (see \Cref{fig:P'-case1}).
			\begin{equation*}
				P' = P_1 u_i Q_i v_i P_3 v_j Q_j u_j P_2.
			\end{equation*}
			We claim that $P'$ satisfies the requirements. Indeed, notice that $P'$ has the same ends as $P$. Moreover, $\ell(P') \ge \ell(P) - r/4 + \ell(Q_i) + \ell(Q_j) + \ell(P_3) \ge \ell(P) + 2$, as at most $r/4$ edges of $P$ are skipped, the paths $Q_i$ and $Q_j$ each have length at least $1$, and $P_3$ has length at least $r/4$. Also, $\ell(P') \le \ell(P) + \ell(Q_i) + \ell(Q_j) + \ell(P_3) \le \ell(P) + r$, using $\ell(Q_i), \ell(Q_j) \le r/4$ and $\ell(P_3) \le r/2$. 
			\begin{figure}[h]
				\vspace{-.5cm}
				\begin{subfigure}[b]{.5\textwidth}
					\centering
					\includegraphics[scale = .8]{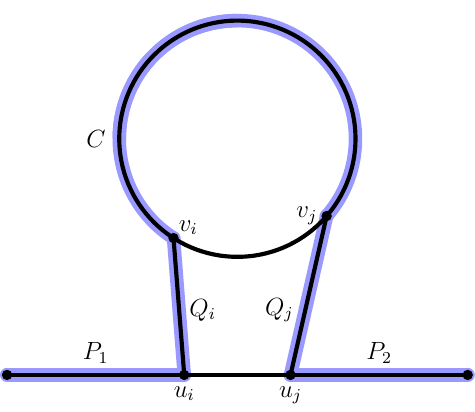}
					\caption{$P'$ when many $P_i$'s are short}
					\label{fig:P'-case1}
				\end{subfigure}
				\begin{subfigure}[b]{.5\textwidth}
					\centering
					\includegraphics[scale = .8]{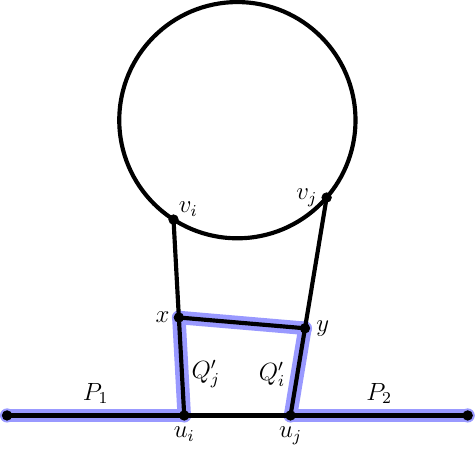}
					\caption{$P'$ when many $P_i$'s are long}
					\label{fig:P'-case2}
				\end{subfigure}
				\caption{the path $P'$}
			\end{figure}

			Now suppose that the latter case holds, namely at least $r/4$ of the paths $P_i$ have length at least $r/4$; let $I$ be the set of indices $i$ such that $\ell(P_i) \ge r/4$. We claim that there is a subpath of $P$ of length at most $r/8$ that contains at least $32\alpha/r$ vertices $u_i$, with $i \in I$.
			Indeed, otherwise
			\begin{equation*}
				\ell \le \left(\frac{\ell}{r/8} + 1\right) \cdot \frac{32\alpha}{r}
				< \frac{2^{9}\alpha}{r^2} \cdot \ell
				< \ell, 
			\end{equation*}
			a contradiction (using $r \ge 32\sqrt{\alpha}$).
			Let $J$ be a subset of $I$ of size at least $32\alpha/r$ such that the vertices $u_j$ with $j \in J$ are on a subpath of $P$ of length at most $r/8$. 
			For each $j \in J$, let $U_j$ be the set of vertices in $P_j$ whose distance from $u_j$ in $P_j$ is \emph{even} and between $r/8$ and $r/4$. Then $U_j$ is an independent set (as $P_j$ is an induced path) of size at least $r/16$. It follows that the union $\bigcup_{j \in J}U_j$ has size at least $(r/16) \cdot (32\alpha/r) > \alpha$. As $\alpha(G) \le \alpha$, there is an edge $xy$ in this union; say $x \in U_i$ and $y \in U_j$ (so $i \neq j$). Let $Q_i'$ be the subpath of $Q_j$ from $u_j$ to $x$, let $Q_j'$ be the subpath of $Q_j$ from $y$ to $u_j$, and let $P_1$ and $P_2$ be the two subpaths of $P$ obtained by removing the segment between $u_i$ and $u_j$, such that $u_i \in V(P_1)$. 
			Define a path $P'$ as follows (see \Cref{fig:P'-case2}).
			\begin{equation*}
				P' = P_1 u_i Q_i' xy Q_j' u_j P_2.
			\end{equation*}
			Again, we claim that $P'$ satisfies the requirements. Indeed, it has the same ends as $P$. Moreover, $\ell(P') \ge \ell(P) - r/8 + \ell(Q_i') + \ell(Q_j') + 1 > \ell(P)$ (using $\ell(Q_i'), \ell(Q_j') \ge r/8$), and $\ell(P') \le \ell(P) + \ell(Q_i') + \ell(Q_j') + 1 \le \ell(P) + r$ (using $\ell(Q_i'), \ell(Q_j') \le r/4$).
		\end{proof}
\section{Lower range} \label{sec:lower}

	In this section we prove \Cref{lem:lower-range}, restated here, about short cycles.
	\lemLowerRange*

	The general outline is very similar to that of the proof of the lower range in \cite{draganic2023chvatal}, but we need to work harder for each step.
	We find very short cycles (of length up to $7$) via ad-hoc arguments; for comparison, in \cite{draganic2023chvatal} the authors needed to find cycles of length up to $5$, and their proofs do not quite carry to our setting. 

	To find longer cycles in this range, our proof splits into two cases. When $n/\alpha \ge \delta \alpha$, we use results about cycle-complete Ramsey numbers (\Cref{thm:ramsey-erdos} and \Cref{thm:ramsey-keevash}), exactly as in \cite{draganic2023chvatal}. Otherwise, we use a result about the Tur\'an number of even cycles due to Bondy and Simonovits \cite{bondy1974cycles} (see \Cref{thm:even-cycle}). 
	This works directly for even lengths, but for odd lengths we need a correction. 
	\Cref{lem:odd-cycles} provides such a correction: it finds a large subgraph $H \subseteq G$ such that every edge $e$ in $H$ can be extended to a cycle of length $3$ or $5$ in $G$, whose vertices outside of $e$ are not in $H$. This is a variant of Lemma 2.7 in \cite{draganic2023chvatal}, where the stronger condition on the connectivity allowed the authors to find such $H$ where every edge $e$ is in a triangle in $G$ whose third vertex is outside of $H$.

	\subsection{Lemmas}

		For technical reasons, we will find very short cycles (of length at most $7$) separately. Since the proof is ad-hoc and quite technical, we present it in an appendix (see \Cref{sec:lower-short-cycles}).

		\begin{lemma} \label{lem:short-cycles}
			Let $G$ be a graph with $\delta(G) > \alpha(G) \gg 1$.
			Then $G$ contains a cycle of length $\ell$, for every $\ell \in \{3,\ldots,7\}$.
		\end{lemma}

		For longer cycles we will consider two cases, depending on whether or not $n/\alpha$ is larger than $\delta \alpha$. If yes, then the proof follows directly from the results we mentioned earlier regarding cycle-complete Ramsey numbers (\Cref{thm:ramsey-erdos,thm:ramsey-keevash}). If not, then we use the following result about the Tur\'an number of even cycles.

		\begin{theorem}[Bondy--Simonovits \cite{bondy1974cycles}] \label{thm:even-cycle}
			Let $\ell \ge 1$ be an integer, 
			and suppose that $G$ is an $n$-vertex graph with $e(G) \ge \max\{20\ell n^{1+1/\ell}, 200n \ell\}$. Then $G$ contains a cycle of length $2\ell$.
		\end{theorem}

		The following lemma allows us to find odd cycles in the latter case, by applying \Cref{thm:even-cycle} to a subgraph $H$ of $G$ which has the property that every edge $e$ is in a short odd cycle in $G$ whose vertices which are not in $e$ are not in $H$.

		\begin{lemma} \label{lem:odd-cycles}
			Let $0 < \delta \ll 1$ and let $\alpha, n \ge 1$ be integers.
			Suppose that $G$ is an $n$-vertex graph with $\delta(G) > \alpha(G) = \alpha$.
			Then there is a subgraph $H$ of $G$ with at least $\delta n \alpha$ edges, such that each edge $e$ in $H$ is contained in a cycle $C$ of length $3$ or $5$ satisfying $V(C) \cap V(H) = V(e)$.
		\end{lemma}

	\subsection{Proof of lower range lemma} \label{sec:lower-short-cycles}

		\begin{proof}[Proof of \Cref{lem:lower-range}]
			Let $\eta$ be a constant satisfying $\delta \ll \eta \ll \alpha$.
			Write $m = \max\{n/\alpha, \delta \alpha\}$, and let $\ell$ be an integer satisfying $3 \le \ell \le m$.
			If $\ell \le 7$ then there is a cycle of length $\ell$, by \Cref{lem:short-cycles}, so we may assume $\ell \ge 8$.
			We consider two cases, according to the value of $m$.

			Suppose first that $m = n/\alpha$.
			Since $G$ has no independent set of size $\alpha+1$ and $\alpha$ is large, there is a cycle of length exactly $\ell$, using aforementioned results about cycle-complete Ramsey numbers, namely \Cref{thm:ramsey-erdos} if $\ell \le \log \alpha$ and \Cref{thm:ramsey-keevash} otherwise.

			Suppose now that $m = \delta \alpha$.
			Apply \Cref{lem:odd-cycles} to find a subgraph $H$ of $G$ with at least $\eta n \alpha$ edges such that for every edge $e \in E(H)$ there is a cycle $C_e$ of length $3$ or $5$ that contains $e$ and whose vertices which are not in $e$ are not in $H$.
			Let $H_i$ be the subgraph of $H$ with edges $e$ such that $|C_e| = i$, for $i \in \{3,5\}$.
			Let $i \in \{3,5\}$ be such that $e(H_i) \ge \eta n \alpha / 2$.
			Write 
			\begin{equation*}
				k = \left\{
					\begin{array}{ll}
						\frac{\ell - i + 2}{2} & \text{$\ell$ is odd} \\
						\frac{\ell}{2} & \text{$\ell$ is even}.
					\end{array}
					\right.
			\end{equation*}
			\begin{claim}
				$e(H_i) \ge \max\{20k n^{1+1/k}, 200nk\}$.
			\end{claim}
			\begin{proof}
				First note that $200nk \le 200 n \ell \le 200 \delta n \alpha \le \eta n \alpha / 2 \le e(H_i)$. 

				Next, note that, as $\ell \ge 8$, we have $k \ge 3$.
				Moreover, since $m = \delta \alpha$, we have $\delta \alpha \ge n / \alpha$, so $\alpha \ge \delta^{-1/2}\sqrt{n}$. This implies $e(H_i) \ge \eta \delta^{-1/2} n^{3/2}/2 \ge n^{3/2}$. 
				Now, if $k \ge \log n$ then $20kn^{1 + 1/k} \le 20 \cdot e \cdot kn \le e(H_i)$. If, instead, $3 \le k \le \log n$, then $20kn^{1 + 1/k} \le 20 \log n \cdot n^{4/3} \le n^{3/2} \le e(H_i)$, using that $\alpha$ and thus $n$ are large.
			\end{proof}

			By the claim and \Cref{thm:even-cycle}, the graph $H_i$ has a cycle of length exactly $2k$. If $\ell$ is even, then this is a cycle of length $\ell$ in $G$. If $\ell$ is odd, then let $C$ be a cycle of length $2k = \ell - (i-2)$ in $H_i$. Let $C'$ be the cycle obtained by replacing one edge $e$ in $C$ by the subpath of $C_e$ obtained by removing $e$. Notice that this is indeed a cycle, as the vertices of $C_e$ which are not in $e$ are not in $H_i$, and thus not in $C$. Also, $|C'| = |C| - 1 + (i-1) = \ell$.
		\end{proof}

	\subsection{Finding many edges in short odd cycles}
		We now prove \Cref{lem:odd-cycles}.
		The idea is to find $\Omega(n\alpha)$ edges in $G$ that are contained in either a triangle or a $5$-cycle. 
		Once that is done, a simple probabilistic argument can be used to find a subgraph $H$ with the desired properties.

		\begin{proof}[Proof of \Cref{lem:odd-cycles}]
			Let $\eta$ be a constant satisfying $\delta \ll \eta \ll 1$.
			For each vertex $u$ let $M(u)$ be a maximum matching in $N(u)$, and let $I(u) := N(u) - M(u)$ (so $I(u)$ is independent).
			We define sets $U_1, \ldots, U_4$ as follows. 
			\begin{itemize}
				\item
					Let $U_1$ be the set of vertices $u$ satisfying $|M(u)| \ge \alpha/12$.
				\item
					Let $U_2$ be the set of vertices $u \notin U_1$ for which there exists a triangle $(uvw)$, where $v \notin U_1$. 
				\item
					Let $U_3$ be the set of vertices $u \notin U_1 \cup U_2$ for which there exist three vertices $v,w,x \neq u$ such that $(vwx)$ is a triangle, $v \notin U_1$, and $uw$ is an edge.
				\item
					Let $U_4 = V(G) - (U_1 \cup U_2 \cup U_3)$.
			\end{itemize}
			We show that $|U_4| \le n/3$.
			Indeed, for each $u \in U_4$ let $T_u$ be a triangle that contains $u$ (such a triangle exists by $\delta(G) > \alpha$). 
			We claim that the triangles $T_u$, with $u \in U_4$, are pairwise vertex-disjoint.
			To see this, consider distinct $u_1,u_2 \in U_4$, and write $T_u = (u_1v_1w_1)$ and $T_{u_2} = (u_2v_2w_2)$. Notice that $v_1,w_1,v_2,w_2 \in U_1$, because $u_1,u_2 \notin U_2$. Thus, if $T_{u_1}$ and $T_{u_2}$ share a vertex, we may assume $v_1 = v_2$. This is a contradiction to $u \notin U_4$, as $(u_2v_2w_2)$ is a triangle, $u_2 \notin U_1$, and $u_1v_2$ is an edge.
			This shows that the triangles $T_u$, with $u \in U_4$, are pairwise vertex-disjoint, which implies the desired inequality $|U_4| \le n/3$.

			Let $\cF$ be a family of ordered triples and quintuples, defined as follows. 
			It will be useful to note that for distinct $u,v \notin U_1$ either $|I(u) \cap I(v)| \ge \alpha/12$ or there is a matching of size at least $\alpha/12$ whose edges have one end in $I(u)$ and the other in $I(v)$; this follows from $I(u)$ and $I(v)$ being independent sets of size at least $2\alpha/3$.
			\begin{itemize}
				\item
					For every $u \in U_1$ and $xy \in M(u)$, add the ordered triples $(u,x,y)$ and $(u, y, x)$ to $\cF$.
				\item
					For every $u \in U_2$, fix $v, w$ such that $v \notin U_1$ and $(uvw)$ is a triangle.
					Because $I(u)$ and $I(v)$ are two independent sets of size at least $2\alpha/3$.
					If $|I(u) \cap I(v)| \ge \alpha/12$, add the ordered triple $(u,x,v)$ to $\cF$ for each $x \in I(u) \cap I(v)$. 
					Otherwise there is a matching $M$ of size at least $\alpha/12$ between $I(u)$ and $I(v)$. Add to $\cF$ the ordered quintuple $(u, x, y, v, w)$, for each edge $xy \in M$ with $x \in I(u)$ and $y \in I(v)$.
				\item
					For every $u \in U_3$, fix $v, w, x \neq u$ such that $(v w x)$ is a triangle, $v \notin U_1$, and $uw$ is an edge. If $|I(u) \cap I(v)| \ge \alpha/12$, add $(u, y, v, x, w)$ with $y \in I(u) \cap I(v)$ to $\cF$. Otherwise, let $M$ be an $I(u)-I(v)$ matching of size at least $\alpha/12$, and add $(u, y, z, v, w)$ to $\cF$, for all $yz \in M$ (where $y \in I(u)$ and $z \in I(v)$). 
			\end{itemize}
			Notice that for every $u \notin U_4$ there are at least $\alpha/12$ ordered tuples in $\cF$ whose first element is $u$, showing that $|\cF| \ge 2n/3 \cdot \alpha/12 = \alpha n/18$. Moreover, for every two vertices $x,y$, there is at most one ordered tuple in $\cF$ whose first element is $x$ and the second $y$.

			Let $X$ be a random set of vertices, obtained by including each vertex with probability $1/2$, independently. 
			Let $\cP$ be the set of ordered pairs $(x,y)$ such that $x,y \in X$ and $\cF$ contains a tuple whose first element is $x$ and the second is $y$, and whose other elements are not in $X$.
			Then $\Ex(|\cP|) \ge |\cF| \cdot 2^{-5} \ge \alpha n / 576$. Fix an outcome of $X$ such that $|\cP| \ge \alpha n / 576$.
			Let $H$ be the graph on vertex set $X$, with edges $xy$ such that at least one of $(x,y)$ and $(y,x)$ is in $\cP$. Then $e(H) \ge |\cP|/2 \ge \alpha n / 1152$. The graph $H$ satisfies the requirements of the lemma.
		\end{proof}

\section{Conclusion} \label{sec:conc}

	In this paper we proved that if a graph $G$ has sufficiently many vertices, and satisfies $\kappa(G) > \alpha(G)$, then $G$ is pancyclic, thereby proving a long standing conjecture of Jackson and Ordaz \cite{jackson1990chvatal} for large graphs.
	Of course, it would be nice to settle the conjecture for all graphs. As mentioned in the introduction, it is plausible that a slight generalisation of Jackson and Ordaz's conjecture holds, namely that every graph $G$ that contains a triangle and satisfies $\kappa(G) \ge \alpha(G)$ is pancyclic. Tools from this paper are likely to be helpful in proving this generalisation (if true), but not everything carries through directly, and we have decided not to pursue this direction as the paper is long enough as it is. 

	As a different avenue for potential future research, we mention a question from \cite{draganic2022pancyclicity}. In \cite{draganic2022pancyclicity}, Dragani\'c, Munh\'a-Correia, and Sudakov proved that if $G$ is Hamiltonian and $n \ge (2+\Omega(1))\alpha^2$, where $n$ is the number of vertices in $G$ and $\alpha$ is $G$'s independence number, then $G$ is pancyclic. While this bound is tight, up to the $\Omega(1)$ error term, the authors suspect that a much weaker bound suffices for guaranteeing a cycle of length $n-1$.

	\begin{question}
		Is there a constant $c$ such that if $G$ is a Hamiltonian graph on $n$ vertices, with independence number $\alpha$ where $n \ge c \cdot \alpha$, then $G$ has a cycle of length $n-1$?
	\end{question}

\bibliography{pancyclic}
\bibliographystyle{amsplain}

\appendix

	\section{Finding short cycles}
		We now prove \Cref{lem:short-cycles}, about finding cycles of length between $3$ and $7$. 
		\begin{proof}[Proof of \Cref{lem:short-cycles}]
			Write $\alpha = \alpha(G)$. For each vertex $u$, let $I(u)$ be a largest independent set in the neighbourhood $N(u)$.

			\paragraph{\boldmath$\ell = 3$.}
				Since $|N(u)| > \alpha$ for every vertex $u$, there is an edge in $N(u)$, showing that $G$ contains a triangle. 

			\paragraph{\boldmath$\ell = 4$.}
				Notice that if, for some vertex $u$, its neighbourhood $N(u)$ contains a path $v_1v_2v_3$ of length $2$, then $(uv_1v_2v_3)$ is a cycle of length $4$. So suppose that $N(u)$ does not contain a path of length $2$, for every vertex $u$. In particular, each neighbourhood $N(u)$ induces a matching, and thus $|I(u)| \ge \alpha/2$.
				Consider a triangle $(uvw)$. If $|I(u) \cap I(v)| \ge 2$, then take a vertex $x \in \big(I(u) \cap I(v)\big) - \{w\}$, and note that $(uxvw)$ is a $4$-cycle (see \Cref{fig:4a}). So suppose that $|I(u) \cap I(v)|, |I(u) \cap I(w)|, |I(v) \cap I(w)| \le 1$.
				Thus $|I(u) \cup I(v) \cup I(w)| \ge 3\alpha/2 - 3 > \alpha$, showing that $I(u) \cup I(v) \cup I(w)$ contains an edge $xy$. By independence of $I(u), I(v), I(w)$, and by symmetry, we may assume $x \in I(u)$ and $y \in I(v)$. Thus $(uxyv)$ is a $4$-cycle (see \Cref{fig:4b}).
				\begin{figure}[h]
					\centering
					\begin{subfigure}[b]{.3\textwidth}
						\centering
						\includegraphics[scale = .7]{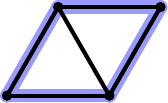}
						\caption{$(uxvw)$}
						\label{fig:4a}
					\end{subfigure}
					\begin{subfigure}[b]{.3\textwidth}
						\centering
						\includegraphics[scale = .7]{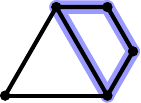}
						\caption{$(uxyv)$}
						\label{fig:4b}
					\end{subfigure}
					\caption{4-cycles}
				\end{figure}

			\paragraph{\boldmath$\ell = 5$.}
				To see that $G$ contains a $5$-cycle, we may assume that $N(u)$ does not contain a path of length $3$, for every vertex $u$. It follows that $|I(u)| \ge \alpha/3$.
				Suppose that some neighbourhood $N(u)$ contains a matching $\{v_1 w_1, \ldots, v_4 w_4\}$ of size $4$.
				If $|I(v_i) \cap I(v_j)| \ge 10$ for some distinct $i,j \in [4]$, pick $x \in \big(I(v_i) \cap I(v_j)\big) - \{v_1, w_1, \ldots, v_4, w_4, u\}$, then $(uv_ixv_jw_j)$ is a $5$-cycle (see \Cref{fig:5a}).
				Otherwise, the set $I(v_1) \cup \ldots \cup I(v_4)$ has size at least $4\alpha/3 - 6\cdot 9 > \alpha$, and so it has an edge $xy$, say $x \in I(v_i)$ and $y \in I(v_j)$. Notice that $i \neq j$, by independence of the sets $I(v_1), \ldots, I(v_4)$. Then $(uv_ixyv_j)$ is a $5$-cycle (see \Cref{fig:5b}).
				Now suppose that no neighbourhood $N(u)$ contains a matching of size $4$, showing that $|I(u)| \ge \alpha - 6$ for every vertex $u$.
				Let $(v_1v_2v_3)$ be a triangle.
				If there is an edge $xy$ with $x \in I(v_i)$, $y \in I(v_j)$, and $x,y \notin \{v_1,v_2,v_3\}$, then, without loss of generality $i = 1$ and $j = 2$, and so $(v_1 xyv_2 v_3)$ is a $5$-cycle (see \Cref{fig:5c}). So suppose no such edge exists. This shows that $|I(v_i) \cup I(v_j)| \le \alpha$ and thus $|I(v_i) \cap I(v_j)| \ge \alpha - 12$ for all $i,j \in [3]$. In particular, there are distinct vertices $x \in I(v_1) \cap I(v_2)$ and $y \in I(v_2) \cap I(v_3)$, such that $x,y \notin \{v_1, v_2, v_3\}$. Then $(v_1 x v_2 y v_3)$ is a $5$-cycle (see \Cref{fig:5d}).
				\begin{figure}[h]
					\centering
					\begin{subfigure}[b]{.29\textwidth}
						\centering
						\includegraphics[scale = .7]{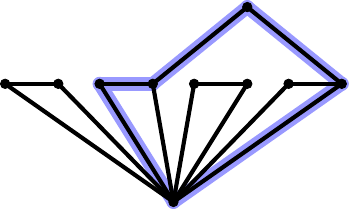}
						\caption{$(uv_ixv_jw_j)$}
						\label{fig:5a}
					\end{subfigure}
					\begin{subfigure}[b]{.29\textwidth}
						\centering
						\includegraphics[scale = .7]{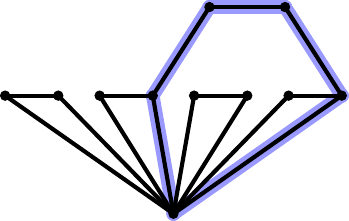}
						\caption{$(uv_ixyv_j)$}
						\label{fig:5b}
					\end{subfigure}
					\begin{subfigure}[b]{.18\textwidth}
						\centering
						\includegraphics[scale = .7]{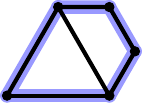}
						\caption{$(v_1xyv_2v_3)$}
						\label{fig:5c}
					\end{subfigure}
					\begin{subfigure}[b]{.18\textwidth}
						\centering
						\includegraphics[scale = .7]{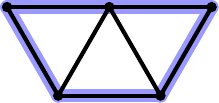}
						\caption{$(v_1xv_2yv_3)$}
						\label{fig:5d}
					\end{subfigure}
					\caption{$5$-cycles}
				\end{figure}

			\paragraph{\boldmath$\ell = 6$.}
				To find a $6$-cycle, we may assume as above that $|I(u)| \ge \alpha/4$ for every vertex $u$.
				Suppose that some neighbourhood $N(u)$ contains a matching $\{v_1w_1, \ldots, v_5w_5\}$ of size $5$. Then, similarly to the previous paragraph, there are distinct $i,j \in [5]$ and either a vertex $x \in \big(I(v_i) \cap I(v_j)\big) - \{v_1, w_1, \ldots, v_5, w_5, u\}$, or two vertices $x \in I(v_i) - \{v_1, w_1, \ldots, v_5, w_5, u\}$ and $y \in I(v_j) - \{v_1, w_1, \ldots, v_5, w_5\}$ such that $xy$ is an edge. This yields a $6$-cycle: $(uw_iv_ixv_jw_j)$ in the first case, and $(uv_ixyv_jw_j)$ in the second case (see \Cref{fig:6a,fig:6b}).
				So we now may assume $|I(u)| \ge \alpha - 8$ for every vertex $u$.
				Consider a triangle $(v_1v_2v_3)$. If $|I(v_i) \cap I(v_j)| \ge 6$ for every distinct $i,j \in [3]$, we may find a $6$-cycle of the form $(v_1 x v_2 y v_3 z)$ (see \Cref{fig:6c}). Otherwise, say $|I(v_1) \cap I(v_2)| \le 6$, and so $|I(v_1) \cup I(v_2)| \ge 2\alpha - 22$. In particular, there is a matching $\{x_1 y_1, x_2 y_2\}$ with $x_1, x_2 \in I(v_1)$ and $y_1, y_2 \in I(v_2)$, yielding the $6$-cycle $(v_1 x_1 y_1 v_2 y_2 x_2)$ (see \Cref{fig:6d}).
				\begin{figure}[h]
					\centering
					\begin{subfigure}[b]{.29\textwidth}
						\centering
						\includegraphics[scale = .6]{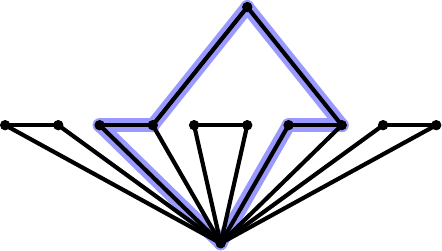}
						\caption{$(uw_iv_ixv_jw_j)$}
						\label{fig:6a}
					\end{subfigure}
					\begin{subfigure}[b]{.29\textwidth}
						\centering
						\includegraphics[scale = .6]{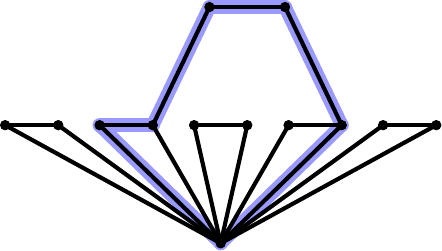}
						\caption{$(uv_ixyv_jw_j)$}
						\label{fig:6b}
					\end{subfigure}
					\begin{subfigure}[b]{.18\textwidth}
						\centering
						\includegraphics[scale = .7]{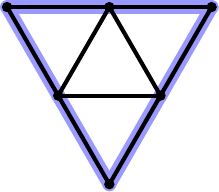}
						\caption{$(v_1xv_2yv_3z)$}
						\label{fig:6c}
					\end{subfigure}
					\begin{subfigure}[b]{.18\textwidth}
						\centering
						\includegraphics[scale = .7]{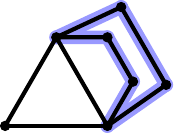}
						\caption{$(v_1x_1y_1v_2y_2x_2)$}
						\label{fig:6d}
					\end{subfigure}
					
					\caption{6-cycles}
				\end{figure}

			\paragraph{\boldmath$\ell = 7$.}
				Finally, we need to prove the existence of a $7$-cycle.
				As usual, we assume $|I(u)| \ge \alpha/5$ for every vertex $u$.
				Suppose $N(u)$ contains a matching $\{v_1w_1, \ldots, v_6w_6\}$. If there is an edge $xy$ with $x \in I(v_i)$ and $y \in I(v_j)$ (and $x,y \notin \{v_1, w_1, \ldots, v_6, w_6, u\}$), then we get a $7$-cycle $(uw_iv_ixyv_jw_j)$ (see \Cref{fig:7a}). So suppose no such edge exists, showing that $|I(v_1) \cup \ldots \cup I(v_6)| \le \alpha$. This implies that, for some distinct $i,j \in [6]$, we have $|I(v_i) \cap I(v_j)| \ge 10$. Let $x_1, \ldots, x_6 \in \big(I(v_i) \cap I(v_j)\big) - \{v_i, v_j, w_j, u\}$. As usual, there are distinct $s,t \in [6]$ such that either there exists $y \in \big(I(x_s) \cap I(x_t)\big) - \{v_i, v_j, w_j, u\}$, or there is an edge $yz$ with $y \in I(x_s) - \{v_i, v_j, w_j, u\}$ and  $z \in I(x_t) - \{v_i, v_j, w_j, u\}$. Either way, there is a $7$-cycle: $(uv_ix_iyx_jv_jw_j)$ in the first case, and $(uv_ix_iyzx_jv_j)$ in the second (see \Cref{fig:7b,fig:7c}). 
				We may now assume that $|I(u)| \ge \alpha - 10$ for every vertex $u$.
				This implies that for every two vertices $a,b$ and any set $S$ of size at most $7$ which does not contain $a$ or $b$, there is a path from $a$ to $b$ that avoids $S$ and has length $2$ or $3$. Consider a triangle $(v_1v_2v_3)$, and for $1 \le i < j \le 3$, let $P_{i,j}$ be a path of length $2$ or $3$ with ends $v_i$ and $v_j$, such that distinct paths have disjoint interiors. If at least two of these paths have length $3$, say $\ell(P_{1,2}), \ell(P_{2,3}) = 3$, then $(v_1 P_{1,2} v_2 P_{2,3} v_3)$ is a $7$-cycle (see \Cref{fig:7d}). If exactly one path has length $3$, say $\ell(P_{1,2}) = 3$, then $(v_1 P_{1,2} v_2 P_{2,3} v_3 P_{1,3})$ is a $7$-cycle (see \Cref{fig:7e}). We may thus assume that all paths have length $2$, i.e.\ there are distinct vertices $u_1, u_2, u_3$ such that $(v_1 u_1 v_2 u_2 v_3 u_3)$ is a $6$-cycle. Let $P$ be a path of length $2$ or $3$, with ends $v_1, u_1$, which avoids $\{v_2, u_2, v_3, u_3\}$. If $P$ has length $2$ then $(v_1 P u_1 v_2 u_2 v_3 u_3)$ is a $7$-cycle, and otherwise $(v_1 P u_1 v_2 u_2 v_3)$ is a $7$-cycle (see \Cref{fig:7f,fig:7g}).
				\begin{figure}[h]
					\centering
					\begin{subfigure}[b]{.4\textwidth}
						\centering
						\includegraphics[scale = .7]{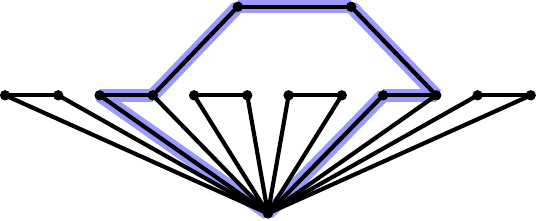}
						\caption{$(uw_iv_ixy_jv_jw_j)$}
						\label{fig:7a}
					\end{subfigure}
					\begin{subfigure}[b]{.23\textwidth}
						\centering
						\includegraphics[scale = .7]{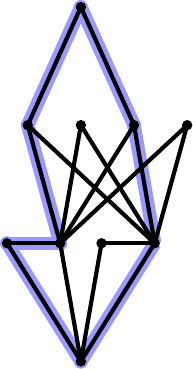}
						\caption{$(uv_ix_iyzx_jv_j)$}
						\label{fig:7b}
					\end{subfigure}
					\begin{subfigure}[b]{.23\textwidth}
						\centering
						\includegraphics[scale = .7]{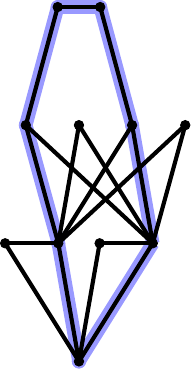}
						\caption{$(uv_ix_iyzx_jv_j)$}
						\label{fig:7c}
					\end{subfigure}
					\vspace{1cm}

					\begin{subfigure}[b]{.23\textwidth}
						\centering
						\includegraphics[scale = .7]{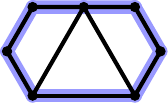}
						\caption{$(v_1P_{1,2}v_2P_{2,3}v_3)$}
						\label{fig:7d}
					\end{subfigure}
					\begin{subfigure}[b]{.23\textwidth}
						\centering
						\includegraphics[scale = .7]{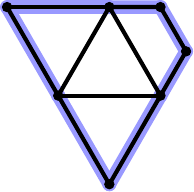}
						\caption{$(v_1P_{1,2}v_2P_{2,3}v_3P_{1,3})$}
						\label{fig:7e}
					\end{subfigure}
					\begin{subfigure}[b]{.23\textwidth}
						\centering
						\includegraphics[scale = .7]{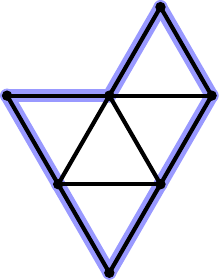}
						\caption{$(v_1Pu_1v_2u_2v_3u_3)$}
						\label{fig:7f}
					\end{subfigure}
					\begin{subfigure}[b]{.23\textwidth}
						\centering
						\includegraphics[scale = .7]{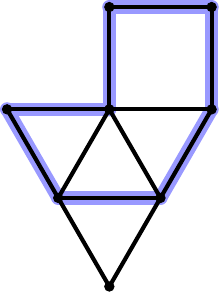}
						\caption{$(v_1Pu_1v_2u_2v_3)$}
						\label{fig:7g}
					\end{subfigure}
					\caption{7-cycles}
				\end{figure}
		\end{proof}

\end{document}